\documentclass{amsart}
\usepackage{booktabs}
\usepackage{todonotes}
\usepackage{calc, color, sidecap}
\usepackage{geometry}
\usepackage{longtable} 
\usepackage{amsmath, amssymb}
\usepackage{amsthm}
\usepackage{amsfonts}
\usepackage{amscd}
\usepackage{array}
\usepackage{bigstrut}
\usepackage{color}
\usepackage{listings}
\usetikzlibrary{decorations.pathreplacing}
\usepackage{enumerate}
\usepackage{pgfplots}
\usepackage{subfigure}
\usepackage{hyperref}

\usepackage{subfigure}
\usepackage{wrapfig}
\usepackage{tikz-cd}
\usepackage{graphicx}
\newcolumntype{L}{>{\centering\arraybackslash}p{3cm}}

\usepackage{pgf,tikz}
\usetikzlibrary{arrows}
\usetikzlibrary{shapes,decorations}

\newtheorem{thm}{Theorem}[section]
\newtheorem{lem}[thm]{Lemma}

\newtheorem{prop}[thm]{Proposition}

\theoremstyle{definition}
\newtheorem{defn}[thm]{Definition}

\theoremstyle{remark}
\newtheorem*{rem}{Remark}

\newcommand{\PP}{\mathbb{P}}
\newcommand{\FF}{\mathbb{F}}

\newcommand{\ZZ}{\mathbb{Z}}
\newcommand{\QQ}{\mathbb{Q}}

\newcommand{\cI}{\mathcal{I}} 
\newcommand{\cE}{\mathcal{E}}

\newcommand{\cC}{\mathcal{C}}
\newcommand{\cB}{\mathcal{B}}
\newcommand{\cD}{\mathcal{D}}

\newcommand{\set}[1]{\left\{#1\right\}}  

\renewcommand{\Im}{\operatorname{Im}}

\DeclareMathOperator{\End}{End}  
\DeclareMathOperator{\Aut}{Aut}  
\DeclareMathOperator{\Pic}{Pic}  
\DeclareMathOperator{\rk}{rank}  
\DeclareMathOperator{\sign}{sign}
\DeclareMathOperator{\GL}{GL}
\DeclareMathOperator{\Fix}{Fix}
\DeclareMathOperator{\disc}{disc}
\DeclareMathOperator{\MW}{MW}

\usepackage{epsfig}

\makeatletter
\def\imod#1{\allowbreak\mkern10mu({\operator@font mod}\,\,#1)}
\makeatother

\definecolor{pistachio}{rgb}{0.58, 0.77, 0.45}
\definecolor{red(munsell)}{rgb}{0.95, 0.0, 0.24}
\definecolor{eggshell}{rgb}{0.94, 0.92, 0.84}

\title{On some K3 surfaces with order sixteen automorphism}

\pgfplotsset{compat=1.12}

\author{Paola Comparin}
\address{
	Departamento de Matem\'atica y Estad\'istica, 
	Universidad de la Frontera, Av. Francisco Salazar 1145,
	Temuco, Chile}
\email{paola.comparin@ufrontera.cl}

\author{Nathan Priddis}
\address{
	Department of Mathematics, 275 TMCB, Brigham Young University,
	Provo, UT 84602, USA.
}
\email{priddis@mathematics.byu.edu}

\author{Alessandra Sarti}
\address{Universit\'e de Poitiers,
	Laboratoire de Math\'ematiques et Applications,
	UMR 7348 du CNRS
	B\^at. H3 - Site du Futuroscope
	TSA 61125,
	11 bd Marie et Pierre Curie,
	86073 Poitiers Cedex 9,
	France
}
\email{alessandra.sarti@math.univ-poitiers.fr}


\begin{document}

\begin{abstract}
We consider K3 surfaces of Picard rank 14 which admit a purely non-symplectic automorphism of order 16. 
The automorphism acts on the second cohomology group with integer coefficients and we compute the invariant sublattice for 
the action. We show that all of these K3 surfaces admit an elliptic fibration and we compute the invariant lattices in a geometric way by using
special curves of the elliptic fibration. The computation of these lattices plays an important role when one wants to study moduli spaces and mirror symmetry for lattice polarized K3 surfaces. 
\end{abstract}

\maketitle

\section{Introduction}

One of the most remarkable traits about K3 surfaces is that we can learn so much about their geometry by studying their cohomology---in particular, by studying the N\'eron-Severi lattice. 
When a K3 surface $X$ admits an automorphism of finite order, it can be used to further understand the geometry of $X$---in particular by the study of the invariant sublattice $S(\gamma)$ of cocycles in $H^2(X,\ZZ)$ that are preserved by $\gamma^*$. 
The invariant lattice is also important in other applications, such as mirror symmetry in the context of theoretical physics, see e.g. \cite{HIS}.

In this article, we will consider certain K3 surfaces that admit a non-symplectic automorphism of order 16. 
In fact, we will show that the K3 surfaces in question, together with their non-symplectic automorphisms, are unique.

The study of automorphisms of K3 surfaces is a relatively old subject, beginning with the study of symplectic automorphisms by Nikulin, Mukai, and others (see \cite{nikulin3, Mu}). 
The study of non-symplectic automorphisms has also been a topic of much interest in the past few years, again beginning with Nikulin, who studied non-symplectic involutions \cite{nikulin2}. Others have studied non-symplectic automorphisms of higher order, including Artebani-Sarti-Taki \cite{other_primes} (prime order),  Sch\"utt \cite{Schutt2010} (order $2^k$), Dillies \cite{order_six} (order 6), Al Tabbaa-Sarti \cite{order_eight} (order 8), and Al Tabbaa-Sarti--Taki \cite{order16} (order 16). Invariant lattices have also been used to understand mirror symmetry for K3 surfaces in \cite{ABS,CLPS,CP}.

As mentioned previously, in this article we will consider K3 surfaces with purely non-symplectic automorphisms of order 16. 
In \cite{order16} Al Tabbaa--Sarti--Taki classified these surfaces,
dividing them into two classes according to the rank of the Picard lattice: the first class has Picard rank 6 and the second has Picard rank 14. In the latter case, the moduli space of such K3 surfaces with the automorphism is 0-dimensional. This raises the question whether  a K3 surface of this type---together with the automorphism---is unique. 

Another natural question that arises is how to determine the invariant lattice. By \cite[Theorems 3.1 and 5.1]{order16} there are five cases of K3 surfaces with Picard rank 14 with a purely non-symplectic automorphism of order 16 and the rank of the invariant lattice for each automorphism is given. Of course, the rank is not enough to determine the lattice, and finding the invariant lattice is in general quite a difficult problem, especially when the order of the non--symplectic automorphism is not prime (see e.g. the discussion in \cite{CP}).  
These five cases are listed in Tables~\ref{tab:thm5.1} and \ref{tab:thm3.1} and described in detail in Sections~\ref{s:table1} and \ref{s:table2}. In this article we will investigate both of these questions. 

Of the two questions raised above, the first---i.e. whether the K3 surfaces with non-symplectic automorphisms are unique---was partially answered by Brandhorst in \cite[Theorem 5.9]{Br}, where he proved that the K3 surfaces in question are unique. In fact, Brandhorst showed that of the 5 cases listed above, there are only three unique K3 surfaces, and these are uniquely determined by their Picard lattice. This is done by studying the transcendental lattice $T_X$ and the ring homomorphisms
\[
\ZZ[x]/(c_n(x))\to O(T^\vee_X/T_X)
\] 
where $c_{n}(x)$ is the cyclotomic polynomial for the primitive $n$-th roots of unity and $T_X^\vee$ denotes the dual lattice of the transcendental lattice $T_X$. 
Since there are only three distinct Picard lattices among the five, we see that two of the three K3 surfaces each admits two different automorphisms of order 16. This is also suggested in the statement of Theorem~\ref{t:tables} (esp. Table~\ref{tab:thm5.1} and the statement preceding Table~\ref{tab:thm3.1}).

For the K3 surfaces in Table~\ref{tab:thm5.1}, we already know that the K3 surfaces together with their automorphisms of order 16 are unique (nonetheless we provide an alternative proof of this fact). Thus the only remaining question is can we determine the invariant lattice, which we do.

As mentioned above, in Table~\ref{tab:thm3.1} there is only one K3 surface to study, but two different types of automorphism. In order to study the two types of automorphisms, we will provide an alternate proof from that given in \cite{Br} that the K3 surface in question is unique. Our method will produce an elliptic fibration of the K3 surface; the advantage to this method, is that we will on the way show that the elliptic fibration is invariant under the action of any automorphism of either type listed in Table~\ref{tab:thm3.1}. This will also allow us to describe the Picard lattice in terms of root lattices, which was not known before. It will also lead the way to the proving that the automorphisms in question are unique. This turns out to be more difficult, as the group of automorphisms for this K3 surface is infinite.

Regarding the second question for the two types of automorphisms in Table~\ref{tab:thm3.1}---namely determining the invariant lattice---as soon as we have a description of this surface as an elliptic fibration, there are two very natural automorphisms of order 16 that arise---one of each type. As we just mentioned, these automorphisms are unique (up to conjugation), and we compute the invariant lattice of these two automorphisms.

Our result is as follows (see Sections~\ref{s:table1} and \ref{s:table2} for notation). 
\begin{thm}
If $X$ is a K3 surface with a purely non--symplectic automorphism $\gamma$ of order 16 such that $\Pic(X) = S(\gamma^8)$ has rank 14 as in one of the lines of Tables~\ref{tab:thm5.1} and \ref{tab:thm3.1}, then the pair $(X,\gamma)$ is unique.

\begin{enumerate}[(i)]
	\item If $(X,\gamma)$ has invariants as in one of the three lines of Table~\ref{tab:thm5.1}, the invariant lattice $S(\gamma)$ with $r=\rk S(\gamma)$ is as in Table \ref{case1} for each surface.

	\item If $(X,\gamma)$ is as in one of the two lines of Table~\ref{tab:thm3.1},
	the invariant lattice $S(\gamma)$ with $r=\rk S(\gamma)$ is as in Table \ref{case2}.
\end{enumerate}

\begin{table}[!htb]
    \begin{minipage}{.5\linewidth}
      \caption{Case $(i)$}
      \label{case1}
      \centering
       \begin{tabular}{c|c}
		$r$& $S(\gamma)$\\
		\hline
		13&$U\oplus E_8 \oplus A_3$ \\
		11&$T_{2,5,6}$\\
		7&$U(2)\oplus D_4\oplus\langle-8\rangle$\\
	\end{tabular}
    \end{minipage}%
    \begin{minipage}{.5\linewidth}
      \centering
        \caption{Case $(ii)$}
        \label{case2}
        \begin{tabular}{c|c}
		$r$& $S(\gamma)$\\
		\hline
		9&$T_{3,4,4}$\\
		7& $U\oplus D_4\oplus \langle -8\rangle$\\
	\end{tabular}\vspace{10pt}
    \end{minipage} 
\end{table}

\end{thm}

\begin{rem}
Two of the invariant lattices we will compute in this article, were also found by the  first two authors in \cite{CP}. There the K3 surfaces were described as (minimal resolutions of) hypersurfaces in weighted projective space, so the method is slightly different. One of the invariant lattices computed there is the second line in Table~\ref{tab:thm5.1}, (which we know is unique) computed in Section~\ref{sss:invariant}. The other invariant lattice computed in \cite{CP} is of the type of the first line of Table~\ref{tab:thm3.1}. Interestingly, the automorphism studied there is a conjugate of the isomorphism studied here. One can show that the automorphisms are in fact different (e.g. they have a different fixed point set; see Definition~\ref{def:unicite}). In fact, viewing this surface as a hypersurface in weighted projective space one can also easily see a non-symplectic automorphism of the type given in the second line of Table~\ref{tab:thm3.1}, which is again a nontrivial conjugate of the automorphism we will consider in the current article. 

\end{rem}

The paper is organized in the following way. 
In Section~\ref{s:background}, we will give a short amount of background regarding K3 surfaces, non--symplectic automorphisms and invariant lattices.
 
In Section~\ref{s:table1}, we will describe the K3 surfaces from Table~\ref{tab:thm5.1}, describing first an elliptic fibration on each surface, and then outlining an alternate proof that each K3 surface is unique. 
The automorphism group for each of the K3 surfaces in this section is known by \cite{Br}, and 
we compute the invariant lattices. 

In Section~\ref{s:table2} we will describe the K3 surface from Table~\ref{tab:thm3.1}  
providing an alternate proof of its uniqueness 
and a description of its Picard lattice. 
We then provide a description of the automorphism group of this surface, and prove the uniqueness of each type of automorphism. Finally, we compute the invariant lattices.

\section*{Acknowledgement}
The first author has been partially supported by Proyecto FONDECYT Iniciaci\'on en Investigaci\'on N. 11190428, Proyecto Anillo ACT 1415 PIA CONICYT and LIA LYSM for the collaboration France-Italy. The first and third authors have been partially supported by Programa de cooperaci\'on cient\'ifica ECOS-ANID C19E06. The third author is partially supported by the ANR project No. ANR-20-CE40-0026-01. 
We warmly thank Simon Brandhorst for useful comments. The anonymous referee provided some crucial details in the proof of unicity of the automorphisms for the third K3 surface that was missing in the first versions of this paper. We would like to thank the referee for providing these ideas and for the careful reading of the paper.

\section{Background}\label{s:background}

Recall that a {\em K3 surface} $X$ is a compact complex surface $X$ having trivial canonical bundle and such that $\dim H^1(X,\mathcal O_X) = 0$. 
Let $H^{2,0}(X)=\mathbb C\omega_X$ be the vector space of holomorphic two forms on $X$
and let $\gamma\in\Aut(X)$ be an automorphism of order 16 acting on $X$. 
The automorphism $\gamma$  is called {\em purely non--symplectic} if the induced action on $H^{2,0}(X)$ is given by $\gamma^*\omega_X=\xi_{16}\omega_X$, 
where $\xi_{16}$ is a primitive 16th root of unity. 
Observe that by \cite[Theorem 3.1]{nikulin3}, a K3 surface which admits a non-symplectic automorphism is projective. 

Given a K3 surface with a non-symplectic automorphism $\gamma$, 
the {\em invariant lattice} is
\[
S(\gamma):=\{x\in H^2(X,\ZZ) : \gamma^*x=x\}
\]
and we will denote its rank by $r$. The invariant lattice embeds primitively into $\Pic(X)$. 

By \cite[Proposition 2.7, 2.9]{order16} the fixed locus of $\gamma$ can contain only rational curves and isolated fixed points
and it is of the form:
\[
\Fix(X)=E_1\cup\ldots\cup E_{k_\gamma}\cup\{p_1,\ldots,p_{N_{\gamma}}\}
\]
with $k_\gamma$ disjoint rational curves and $N_\gamma$ isolated fixed points.

In the following theorem we summarize the results of \cite[Theorems 3.1 and 5.1]{order16}:

\begin{thm}[\cite{order16}] \label{t:tables}
Let $X$ be a K3 surface and $\gamma$ be a purely non-symplectic automorphism of order 16 on X.
Assume that $\Pic(X)=S(\gamma^8)$ has rank 14.
Then one of the following distinct cases holds:
\begin{itemize}
 
\item The fixed locus $\Fix(\gamma^8)$ contains a curve $C$ of genus $g(C)=2$ or 3. If 
$N'$ is the number of fixed points contained in $C$ then the invariants are as in Table~\ref{tab:thm5.1}.

\begin{table}[h!]
\caption{Invariants when $\Fix(\gamma^8)$ contains a curve $C$ of genus 2 or 3}
	\begin{tabular}{c|cccc|c}
		$r$&$N_\gamma$ &$k_\gamma$ & $N'$ & $g(C)$& $\Pic(X)$\\
		\hline
		13& 12&1&2&3&$U\oplus D_4\oplus E_8$\\
		11&10&1&2&2&$U(2)\oplus D_4\oplus E_8$\\
		7&4&0&2&2&$U(2)\oplus D_4\oplus E_8$\\
	\end{tabular}
\label{tab:thm5.1}
\end{table}

\item There exists an elliptic curve $C$ in the fixed locus of $\gamma^4$ and $\gamma$ acts as an automorphism of order 4 on $C$.  
The curve $C$ determines an elliptic fibration with an invariant reducible fiber of type $IV^*$.
Invariants are as in one of the lines of Table~\ref{tab:thm3.1}.
\begin{table}[h!]
\caption{Invariants when $\Fix(\gamma^4)$ contains an elliptic curve}
	\begin{tabular}{c|ccc}
		$r$ &$N_\gamma$ &$k_\gamma$ \\
		\hline
		9&8&1\\
		7&6&0
	\end{tabular}%
\label{tab:thm3.1}
\end{table}

\end{itemize}
\end{thm}

In \cite[Theorem 6.6]{Br} we learn that there are three unique K3 surfaces with a purely non--symplectic automorphism of order 16 such that $\Pic(X)=S(\gamma^8)$ and is of rank 14. One can distinguish them according to the order of the discriminant group of their Picard lattice. 
The Picard lattices satisfy $|\disc(\Pic(X))|=2^k$, with $k\in\set{2,4,6}$. 
Observe that the K3 surfaces in Table~\ref{tab:thm5.1} have discriminant group of order $2^2$ and $2^4$, according to the Picard lattices listed in the Table. In Table~\ref{tab:thm3.1}, one might notice that the Picard lattice is not given; it will be computed in Section~\ref{s:table2} (see Theorem~\ref{lem:PciY}) where we will see that this surface has discriminant group of order $2^6$ (see also \cite{Br}).

Through the paper, for quadratic forms we will use the notation of \cite{Belc} and results of \cite{nikulin}. 
The definition of the discriminant quadratic forms $\omega_{p,k}^\epsilon, u_k,v_k$ is as follows:
\begin{enumerate}
\item for a prime $p\neq 2$, an integer $k\geq 1$ and $\epsilon\in\set{\pm 1}$, 
we define $\omega_{p,k}^\epsilon:\ZZ/p^k\ZZ\rightarrow \QQ/2\ZZ$ 
via $\omega_{p,k}^\epsilon(1)=ap^{-k}$, where $a$ is the smallest even integer that has $\epsilon$ as quadratic residue modulo $p$.
\item for $p=2$, an integer $k\geq 1$ and $\epsilon\in\set{\pm 1,\pm 5}$, we define $\omega_{2,k}^\epsilon:\ZZ/2^k\ZZ\rightarrow \QQ/2\ZZ$ via 
$\omega_{2,k}^\epsilon(1)=\epsilon 2^{-k}$.
\item for an integer $k\geq 1$, we define $u_k,v_k$ on $\ZZ/2^k\ZZ\times\ZZ/2^k\ZZ$ by the matrices
\[u_k=2^{-k}\left(\begin{array}{cc}
0&1\\
1&0\\
\end{array}\right),
\quad 
v_k=2^{-k}\left(\begin{array}{cc}
2&1\\
1&2\\
\end{array}\right).\]
Observe that $u$ and $v$ will be used to denote $u_1$ and $v_1$, respectively.
\end{enumerate}

The following result allows us to determine uniquely an even lattice from its discriminant form.
Given a discriminant quadratic form $q$ (defined on a finite abelian group)
and a prime $p$, we denote by $q_p$ the restriction of $q$ to the $p$-component $(A_q)_p$ of the discriminant group $A_q$ of $q$. Given an even lattice $L$ (all lattices in this work are considered to be even), we denote by $l(A_L)$ the length of $L$, i.e. the minimal number of generators of the discriminant group $A_L$, and by $(t_+,t_-)$ its signature (over the real numbers). We denote moreover by $q$ the discriminant quadratic form associated to the bilinear form on $L$ and by $\sign(q)$ the signature of the bilinear form. We recall the following useful theorem of Nikulin:

\begin{thm}[ {\cite[Corollary 1.13.3]{nikulin}}]\label{unicity}
An even lattice $L$ with invariants $(t_+, t_-,q)$ exists and is unique if $t_+-t_-\equiv \sign q\ (mod\ 8)$, the sum $t_++ t_-\geq 2+l(A_q)$, and $t_+, t_-\geq 1$.
\end{thm}

In this theorem, the uniqueness should be understood up to isomorphism. All of the lattices in this article satisfy the conditions of Theorem~\ref{unicity}. 

We also recall that the lattice $T_{p,q,r}$, with $p,q,r\in\ZZ$, is the root lattice whose Dynkin diagram has the form of a T and $p,q,r$ are the lengths of the three legs (see \cite{Belc}).
For example $T_{2,5,6}$ is pictured in Figure~\ref{T256}. These lattices appear to some degree in the literature (see e.g. \cite{Belc}, \cite{brieskorn}, \cite{dolgachev3}, \cite{gabrielov}, \cite{pinkham}). 
\begin{figure}[ht]
\caption{Graph for $T_{2,5,6}$.}\label{T256}
\begin{tikzpicture}[xscale=.6,yscale=.5, thick]
\begin{scope}[every node/.style={circle, draw, fill=black!50, inner sep=0pt, minimum width=4pt}]
	\draw (4,8) node[name=e3]{} -- (4,7) node{};
	
	\draw (0,8) node{}
	-- (1,8) node{}
	-- (2,8) node{}
	-- (3,8) node{}
	-- (e3);

\draw (e3) -- (5,8) node{}
	--(6,8) node{}
	--(7,8) node{}
	--(8,8) node{}
	--(9,8) node{};

\end{scope}
	
\end{tikzpicture}
\end{figure}

\begin{defn}\label{def:unicite}
Fix a primitive $16$th root of unity $\xi_{16}$. Let $X$ be a K3 surface and let $\gamma$ be an automorphism of order 16 as in the lines of Table~\ref{tab:thm5.1} and Table~\ref{tab:thm3.1} such that $f^*\omega_X=\xi_{16}\omega_X$.
We say that 
$(X,\gamma)$ is unique if for a K3 surface $Y$ and an automorphism $\varphi$ in the same line of the Tables~\ref{tab:thm5.1} and \ref{tab:thm3.1} 
as $X$ and $\gamma$, there exists an isomorphism $f:X\longrightarrow Y$ such that  $\gamma= f^{-1}\varphi f$.
\end{defn}

We will often make use of the following theorem. 

\begin{thm}[see \cite{order4},\cite{order_six}]\label{thm:tree}
	Let $R_1, R_2,\dots , R_s$ be a tree of smooth rational curves on a K3 surface $X$ and $\gamma$ a non-symplectic automorphism of finite order on $X$ leaving each of the $R_i$ invariant. Then the intersection points of the $R_i$'s are fixed by the automorphism and it suffices to know the action at one intersection point to know the action on the entire tree.
\end{thm}

\section{K3 surfaces from Table~\ref{tab:thm5.1}}\label{s:table1}

As mentioned above, there are three K3 surfaces with purely non-symplectic automorphism of order 16 with Picard rank 14 satisfying $\Pic(X)=S(\gamma^8)$.
We will consider two of them in this section, namely those listed in Table~\ref{tab:thm5.1}. We begin with the smallest discriminant. 

\subsection{The first K3 surface}\label{sec31}
Let $X_2$ be a K3 surface with a purely non-symplectic automorphism $\sigma$ of order 16, such that $\rk \Pic (X_2)=14$ and $|\disc\Pic (X_2)|=2^2$. 
This is the K3 surface on line 1 of Table~\ref{tab:thm5.1} (or line 1 of the table in \cite[Theorem 5.1]{order16}). We know in this case that $\Pic (X_2)=U\oplus E_8\oplus D_4$. 

In \cite[Theorem 6.6, Theorem 7.2]{Br}, Brandhorst has shown that such a K3 surface is unique and that $\Aut(X_2)\cong \ZZ/16\ZZ$. 
According to \cite{Br} and \cite{order16}, the K3 surface admits an elliptic fibration 
\[y^2=x^3+t^2x+t^7\]
and the automorphism $\sigma$ is given by 
\[
\sigma:(x,y,t)\mapsto (\xi_{16}^2x, \xi_{16}^{11} y, \xi_{16}^{10}t).
\]

This elliptic fibration has a fiber of type $I_0^*$ (an extended $D_4$) over $t=0$ and a fiber of type $II^*$ (an extended $E_8$) over $t=\infty$. 
This gives us a configuration of rational curves on $X_2$ given in Figure~\ref{fig:fibX2}. The dotted line represents a section and its intersections with reducible fibers is as indicated, since sections intersect simple components of the reducible fibers. There is also a genus 3 curve fixed by the involution $\sigma^8$ not pictured. 
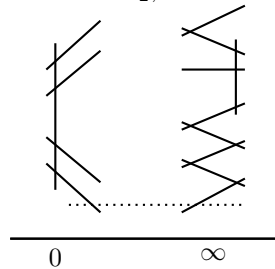
\begin{figure}[ht]
\caption{Fibration on $X_2$, the dotted line is a section}
\begin{tikzpicture}[xscale=.6,yscale=.5, thick]

\draw [thick] (0,-1)--(6,-1);
\draw (1,0.3)--(1,4.2);

\draw (0.8,1.0)--(2,-0.3);
\draw (0.8,1.7)--(2,0.5);
\draw (0.8,2.8)--(2,4);
\draw (0.8,3.5)--(2,4.8);

\node [below] at (1,-1){$0$};
\draw (5,2.3)--(5,4.3);

\draw (3.8,-0.3)--(5.2,0.6);
\draw (3.8,1.1)--(5.2,0.4);
\draw (3.8,0.9)--(5.2,1.6);
\draw (3.8,2.1)--(5.2,1.4);
\draw (3.8,1.9)--(5.2,2.6);

\draw (3.8,3.5)--(5.2,3.5);

\draw (3.8,4.6)--(5.2,3.9);
\draw (3.8,4.4)--(5.2,5.2);

\draw [dotted,thick] (1.3,-0.1)--(5.2,-0.1);

\node [below] at (4.5,-1){$\infty$};

\draw [thick] (0,-1)--(6,-1);

\end{tikzpicture}
\label{fig:fibX2}
\end{figure}

The automorphism $\sigma$ acts with order 8 on the base of $\pi$, and permutes two of the irreducible components of the $I_0^*$ fiber and it preserves all the other curves in the fiber $I_0^*$ and $II^*$. One can check that this action fixes one rational curve (the central curve of the $II^*$) and 12 isolated points, as indicated in Table~\ref{tab:thm5.1}. 

\subsubsection*{Alternate proof of uniqueness of $X_2$} Alternatively, one can prove $X_2$ is unique by the following (geometric) argument. Let $\sigma$ denote any automorphism of order $16$ acting on $X$ (not necessarily unique).  
We first prove a preparatory Lemma.
\begin{lem}\label{fibr_inv}
Let $X$ be a K3 surface with a purely non--symplectic automorphism $\sigma$ of order $2n$, such that $\sigma^n$ is an involution acting as the identity
on $\Pic(X)$. Let $f\in \Pic(X)$ be the class of an elliptic curve and assume moreover that $f$ defines a jacobian elliptic fibration with a section $s$ which is pointwise fixed by $\sigma^n$. Assume moreover that $Fix(\sigma^n)$ does not contain a curve of genus 1. We have
\begin{enumerate}
\item If $R\in\Pic(X)$ such that $\sigma(R)=R$ and $R\cdot f=0$ then $R\cdot \sigma(f)=0$.
\item If $f\cdot \sigma(f)\neq 0$ then $f\cdot \sigma(f)$ is even. This is also true if instead of $\sigma(f)$ we consider a curve 
$C\in\Pic(X)$  not in the fixed locus of $\sigma^n$ such that $\sigma^n(C)=C$, i.e. $f\cdot C\neq 0$ implies $f\cdot C$ is even.
\end{enumerate}
\end{lem}
\begin{proof}
First, we have $0=R\cdot f=\sigma(R)\cdot \sigma(f)=R\cdot \sigma(f)$. This proves the first statement. 

Now for the second statement, assume that $f\cdot \sigma(f)$ contains $2k+1$ points. Let $F$ be a fiber in the elliptic fibration determined by $f$. Observe that $\sigma^n(F)=F$ since $\sigma^n$ acts as the identity on the section and $\sigma^n(f)=f$ since $\sigma^n$ acts as the identity on $\Pic(X)$. Moreover we have also $\sigma^n(\sigma(F))=\sigma(\sigma^n(F))=\sigma(F)$. This means that $\sigma^n$ acts on the set of 
$2k+1$ points in the intersection of $F$ and $\sigma(F)$, but $\sigma^n$ is an involution so that we have a fixed point in the intersection of $F$ with $\sigma(F)$. If we take another generic fiber $F'\in f$ we have also that $F'$ meets $\sigma(F)$ in $2k+1$ points (with multiplicity) since all the fibers in the fibration $\sigma(f)$ are equivalent. For the same reason as before $\sigma^n$ fixes one point in the intersection. We find that $\sigma^n$ contains an infinite 
number of fixed points on $\sigma(F)$, this means that $\sigma(F)$ is a fixed elliptic curve for $\sigma^n$ but this is impossible by the assumption on the fixed locus of $\sigma^n$. 
\end{proof}

Now assume $\Pic(X)=S(\sigma_8)=U\oplus E_8\oplus D_4$ as in the first line of Table~\ref{tab:thm5.1}. We obtain a fibration on $X$ with a fiber of type $II^*=\tilde{E_8}$ and a fiber of type $I_0^*=\tilde{D_4}$ 
which we know it is $\sigma^8$--invariant.  
We show that it is also $\sigma$--invariant by using Lemma \ref{fibr_inv}. 

For this observe that the automorphism $\sigma^8$ fixes a unique curve $C$ of genus $3$ which is hence $\sigma$--invariant. Since the section $s$ is fixed by $\sigma^8$ and each fiber contains $4$ fixed points, the curve $C$ is a 3-section of the fibration. 
One can use now \cite[Lemma 5]{order4} to compute that $f\cdot \sigma(f)\leq 3$. By Lemma \ref{fibr_inv} we can exclude the cases where $f\cdot\sigma(f)$ is odd. We will now show that $f\cdot \sigma(f)=2$ cannot happen, from which we can then deduce that $f\cdot\sigma(f)=0$.

Assume that $f\cdot \sigma(f)=2$. This means that if $F\in f$ then $\sigma(F)$ is a 2-section of the fibration induced by $f$ and the two points in the intersection of $\sigma(F)$ 
with the fibers of the fibration determined by $f$ are exchanged by $\sigma^8$ (since $\sigma^8$ does not fix an elliptic curve); for the generic fiber these two points are distinct. 

Let us look how the 2-section $\sigma(F)$ meets the singular fiber $II^*$.  We denote the external components as follows (see Figure~\ref{fig:II*fiber}): $\mathcal{E}_1$ is the fiber of multiplicity $2$, $\mathcal{E}_8$ is the fiber with multiplicity 1, $\mathcal{E}_9$ is the fiber of multiplicity $3$. Observe that $\sigma(F)$ can not meet the (internal) components that have multiplicity bigger than $2$. 

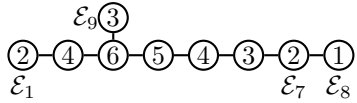
\begin{figure}[ht]
	\caption{Dynkin diagram for the $II^*$ fiber. Each vertex is labelled with the multiplicity of the corresponding curve.}\label{fig:II*fiber}
	\begin{tikzpicture}[xscale=.6,yscale=.5, thick]
	\begin{scope}[every node/.style={circle, draw, inner sep=1pt, minimum width=4pt}]
	
	\node [name=e1] at (1,1){2};
	\node [name=e2] at (2,1){4};
	\node [name=e3] at (3,1){6};
	\node [name=e4] at (4,1){5};
	\node [name=e5] at (5,1){4};
	\node [name=e6] at (6,1){3};
	\node [name=e7] at (7,1){2};
	\node [name=e8] at (8,1){1};
	\node [name=e9] at (3,2){3};
	\draw (e3) -- (e9);
	
	\draw (e1) -- (e2) -- (e3) -- (e4) -- (e5) -- (e6) -- (e7) -- (e8);

	\end{scope}

	\node [below] at (1,.7){$\cE_1$};
	\node [below] at (8,.7){$\cE_8$};
	\node [below] at (7,.7){$\cE_7$};
	\node [left] at (2.9,2){$\cE_9$};

	\end{tikzpicture}
\end{figure}

There is an internal component of multiplicity $2$ (call it $\cE_7$) which is fixed by $\sigma^8$ and $\sigma$ preserves this curve (no curves fixed by $\sigma^8$ are exchanged by $\sigma$; see \cite[Section 5]{order16}). Observe that $F\cdot \cE_7=0$ since $\cE_7$ is contained in a fiber. Now by Lemma~\ref{fibr_inv} we have \[
\sigma(F) \cdot \cE_7=\sigma(F)\cdot \sigma(\cE_7)=0,
\] 
since $\cE_7$ is $\sigma$-invariant. So $\sigma(F)$ cannot meet $\cE_7$. 

The same argument applies to the external component $\mathcal{E}_1$ of multiplicity $2$. Clearly $\sigma(F)$ cannot meet the component $\mathcal{E}_9$, because it is a 2-section and that fiber has multiplicity three. We are left with only the possibility that $\sigma(F)$ intersects $\cE_8$ in two points.   

Let us now consider the elliptic fibration determined by $\sigma(f)$. By the previous discussion, this means that the rational curve $\mathcal{E}_8$ is a 2-section of this
fibration. We have hence a $2:1$ map $\mathcal{E}_8 \rightarrow \mathbb P^1$ which has necessarily 2 ramification points, i.e. there are two fibers in the fibration $\sigma(f)$ which intersects $\mathcal{E}_8$ in only one point, which is necessarily a fixed point for the action of $\sigma^8$ on 
$\mathcal{E}_8$, since $\sigma^8$ preserves all these curves. But $\sigma^8$ acts as an involution on  $\mathcal{E}_8$ (one has to check 
the action of $\sigma^8$ on the singular fiber of the fibration to identify the fixed rational curves: the $II^*$ fiber from the first fibration contains four of them) with two fixed points: one is the intersection point with the internal component of multiplicity 2 and the other is the intersection point with the section $s$. 

Since any fiber of $\sigma(f)$ meets the curve $\mathcal{E}_8$ with multiplicity $2$, then by Hurwitz's formula we find that we have two fibers (reducible or not) of the fibration that meet $\mathcal{E}_8$ in one point of multiplicity two. Since $\sigma^8$ preserves each of them, these should meet $\mathcal{E}_8$ in the intersection point of $s$ with $\mathcal{E}_8$. Observe that they can not meet $\mathcal{E}_8$ in the intersection point of $\mathcal{E}_8$ with $\mathcal{E}_7$ since then they would cut $II^*$ with multiplicity $3$. This means that these two fibers of the fibration $\sigma(f)$ intersect, which is impossible. This shows the following (with the previous notations):
\begin{prop}
The fibration induced by $U\oplus E_8\oplus D_4$ is $\sigma$--invariant.
\end{prop}

We now compute the Weierstrass equation 
in order to conclude the proof of unicity. A priori, we could have more K3 surfaces belonging to the same family with such a non--symplectic automorphism.

We need to understand the action of $\sigma$ on the base of the fibration. Observe that $S(\sigma^8)$ has rank 14 and length 2, so by a result of Nikulin \cite{nikulin2} 
the involution $\sigma^8$ preserves
the fibration and it fixes a curve of genus three and 6 rational curves. 
By checking the action on the fibration (e.g. noticing that the central component of the $II^*$ fiber is preserved, and then using Thereom~\ref{thm:tree})
one can see that the section must be a fixed curve for the action of $\sigma^8$. A similar argument for $\sigma^4$ shows that the section is not pointwise fixed by $\sigma^4$. Thus $\sigma$ acts with order 8 on the base, and so the involution $\sigma^8$ acts on each fiber as $y\mapsto -y$.

The automorphism $\sigma$ has 2 fixed points on the base $\mathbb P^1$, we can assume that these are $0$ and $\infty$. Assume $I_0^*$ is over $t=0$ and $II^*$ over $t=\infty$.

The Weierstrass equation has the form
$y^2=x^3+A(t)x+B(t),\ t\in\PP^1$. Since this is a K3 surface, $A(t),B(t)$ can be thought of homogeneous polynomials of total degree 8 and 12, respectively. 
The discriminant is $\Delta(t)=4A(t)^3+27B(t)^2$ and it has total (homogeneous) degree 24. 
We will supress the second homogeneous variable, and think of these as only polynomials in $t$.

The type of singular fibers is determined by the vanishing order of $A,B,\Delta$.
Following \cite{Mir}, we denote $a(t_0)$ (resp. $b(t_0), \delta(t_0)$) the order of vanishing of $A(t)$  (resp. $B(t), \Delta(t)$) at $t_0$. 
Observe that by \cite[Table IV.3.1]{Mir}, the discriminant $\Delta$ has order of vanishing $10$ at $t=\infty$ and $6$ at $t=0$ for this elliptic fibration. 

Since the Euler number of the fibers is 6 for $I_0^*$ and 10 for $II^*$ we are left with 8 and since there are no more reducible curves we can have either $8I_1$ or $4II$.
The latter is not possible, since the action of $\sigma$ of the basis has order 8, thus there can not be orbits of length 4. 
Thus $\Delta= t^6\cdot R(t)$ with $R(t)$ a polynomial of degree 8 with simple roots and not vanishing in $t=0$. The automorphism $\sigma$ permutes the $8I_1$.
Moreover, the degree of $B(t)$ has to be 7 (since by \cite{Mir} $b(\infty)=5$) and the degree of $A(t)$ is less or equal than 4 since $a(\infty)\geq4$.

By $a(0)=2$ or $\geq 3$ and $b(0)=3$ or $\geq 4$ we see 
\[
A(t)=t^2\cdot P(t),\quad B(t)=t^3\cdot Q(t)
\]
with $\deg(P)\leq2$ and $\deg(Q)=4$. Assume that the action on the basis is $t\mapsto \zeta_8 t$. After applying this transformation to $Q(t)$ we must have $Q(t)\mapsto \zeta_8^l Q(t)$ for some power $l$ of $\zeta_8$. Since the coefficient of $t^4$ can not be zero in $Q(t)$ we find that 
$Q(t)=c_4\cdot t^4$ for a non zero constant $c_4$ which we can suppose equal to $1$. This forces $A(t)$ to disappear exactly to the order 2 on $t=0$, which gives $P(t)=c_0$ for some non zero constant $c_0$ which we can suppose again equal to $1$. So we get 
\[
y^2=x^3+t^2x+t^7.
\]

\subsubsection{Invariant Lattice}
The next step is to calculate the invariant lattice for this automorphism. 
In order to describe the invariant lattice, we first describe a set of curves from the configuration in Figure~\ref{fig:fibX2} that generate $\Pic (X_2)$, and from these we build a generating set for $S(\sigma)$. To do this we first need to label the curves in the configuration in Figure~\ref{fig:fibX2}. 

Let $\cD_1, \ldots, \cD_5$ denote the irreducible components of the $I_0^*$ fiber, with $\cD_1$ being the central curve, $\cD_2$ the curve intersecting the section, and $\cD_3$ and $\cD_4$ the two curves permuted by $\sigma$.
Similarly, let $\cE_1,\ldots, \cE_9$ denote the irreducible components of the $II^*$ fiber (we do not specify which curve is which, as each of them is required to generate every lattice we consider), and let $S$ denote the section.
The genus 3 curve from Theorem~\ref{t:tables} intersects each of $\cD_3, \cD_4$ and $\cD_5$, but this will be irrelevant for our computations. 

We can represent this configuration with the incidence graph in Figure~\ref{fig:X2graph}. The vertex represented by a double circle represents the section $S$. 

\begin{figure}[ht]
\caption{Configuration of curves on $X_2$. The node with two circles represents the section.}

\centering
\begin{tikzpicture}[xscale=.6,yscale=.5, thick]
	
\begin{scope}[every node/.style={circle, draw, fill=black!50, inner sep=0pt, minimum width=4pt}]
	\draw (4,8) node{} -- (4,7) node{}
	-- (4,6) node[name=e3]{}
	-- (4,5) node{}
	-- (4,4) node{}
	-- (4,3) node{}
	-- (4,2) node{}
	-- (4,1) node[name=e9]{};

	\draw (e3) -- (3,6) node{};
	\draw (e9) -- (3,1) node[double, name=section]{}
	-- (2,1) node{}
	-- (1,1) node[name=d1]{}
	-- (0,1) node{};
	\draw (2,2) node{} -- (d1) -- (0,2) node{};
\end{scope}

\end{tikzpicture}
\label{fig:X2graph}
\end{figure}

Since $\rk \Pic (X_2)=14$, at least one of these fifteen curves is redundant. From the theory of elliptic fibrations, the $II^*$ fiber is equivalent to the $I_0^*$ fiber. Thus we can leave out a component of multiplicity one of the $I_0^*$ fiber. 
The remaining curves generate a lattice of rank 14, so the set $\{\cE_1,\dots, \cE_9, S, \cD_1,\dots,\cD_4\}$ is a minimal set of generators for the lattice, as illustrated in Figure~\ref{fig:X2Picbasis}. We have drawn the redundant curve as an empty circle, and the corresponding intersection as a dotted line. 

Also from the theory of elliptic fibrations, one can see that this is the lattice $U\oplus D_4\oplus E_8$. Alternatively, considering the incidence matrix, one can compute the discriminant quadratic form of the lattice generated by these curves. 
This can be done, for example, using the computer algebra system MAGMA, and the function \verb|disc| written in the appendix of \cite{CP}. From this computation, we see that the discriminant quadratic form is $v$ and the rank is 14. Thus by Theorem~\ref{unicity}, this lattice is $U\oplus E_8\oplus D_4$, which is the same as $\Pic(X_2)$. Most importantly, we have exhibited a minimal set of generators for $\Pic(X_2)$. 

\begin{figure}[ht]
\caption{Graph of a minimal set of generators for $\Pic(X_2)$. The filled vertices form a minimal set of generators and the empty vertex is redundant.}
\centering
\begin{tikzpicture}[xscale=.6,yscale=.5, thick]
\begin{scope}[every node/.style={circle, draw, fill=black!50, inner sep=0pt, minimum width=4pt}]
	\draw (4,8) node{} -- (4,7) node{}
	-- (4,6) node[name=e3]{}
	-- (4,5) node{}
	-- (4,4) node{}
	-- (4,3) node{}
	-- (4,2) node{}
	-- (4,1) node[name=e9]{};
	
	\draw (e3) -- (3,6) node{};
	\draw (e9) -- (3,1) node[name=section]{}
	-- (2,1) node{}
	-- (1,1) node[name=d1]{}
	-- (0,1) node{};
	\draw (d1) -- (0,2) node{};
\end{scope}
	
	\node[style=circle, draw, inner sep=0pt, minimum width=4pt](d5) at (2,2) {};
	\draw[dotted] (d5) -- (d1);
	
\end{tikzpicture}
\label{fig:X2Picbasis}
\end{figure}
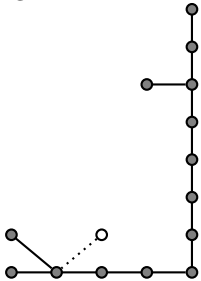

In order to compute the invariant lattice, we consider divisors invariant under $\sigma$. 
Let $L_\cB$ denote the lattice generated by the set $\cB=\{\cE_1,\dots, \cE_9, S, \cD_1,\cD_2, \cD_3+\cD_4, \cD_5\}$, as pictured in the (weighted) graph in Figure~\ref{fig:LBbasis}. The last square vertex represents $\cD_3+\cD_4$ and so has self-intersection $(\cD_3+\cD_4)^2=-4$. The intersection with $\cD_1$ gives $\cD_1\cdot(\cD_3+\cD_4)=2$ as indicated by the weight on the edge of the graph. 
Since $\sigma$ permutes $\cD_3$ and $\cD_4$, but leaves the other curves invariant, we see that $L_\cB\subseteq S(\sigma)$ and we will see that they have the same rank. We will show that these lattices are equal by showing that the embedding $L_\cB\hookrightarrow \Pic(X_2)$ is a primitive embedding.

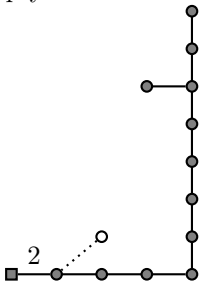
\begin{figure}[ht]
\caption{Graph of generators of $L_\cB$.The square vertex represents the divisor $\cD_3+\cD_4$, and the empty dot is redundant.}

\centering
\begin{tikzpicture}[xscale=.6,yscale=.5, thick]
\begin{scope}[every node/.style={circle, draw, fill=black!50, inner sep=0pt, minimum width=4pt}]
	\draw (4,8) node{} -- (4,7) node{}
	-- (4,6) node[name=e3]{}
	-- (4,5) node{}
	-- (4,4) node{}
	-- (4,3) node{}
	-- (4,2) node{}
	-- (4,1) node[name=e9]{};
	
	\draw (e3) -- (3,6) node{};
	\draw (e9) -- (3,1) node[name=section]{}
	-- (2,1) node{}
	-- (1,1) node[name=d1]{};
\end{scope}
	
	\node[style=rectangle, draw, inner sep=0pt, minimum width=4pt,minimum height=4pt, fill=black!50,](d34) at (0,1) {};
	\draw (d34) -- (d1);
	\node  at (.5,1.5) {2};
	
	\node[style=circle, draw, inner sep=0pt, minimum width=4pt](d5) at (2,2) {};
	\draw[dotted] (d5) -- (d1);
	
\end{tikzpicture}
\label{fig:LBbasis}
\end{figure}

From Table~\ref{tab:thm5.1}, we see that $\rk S(\sigma)=13$. By a computation with the intersection matrix similar to that described above, we see that in fact $\cD_5$ is redundant in the set of generators of $L_\cB$ and that 
\[
\{\cE_1,\dots, \cE_9, S, \cD_1,\cD_2, \cD_3+\cD_4\}
\]
is a minimal set of generators for $L_\cB$, as depicted in Figure~\ref{fig:LBbasis}, so that $\rk L_\cB=\rk S(\sigma)$. 

From the description of the minimal set of generators for $L_\cB$ and $\Pic(X_2)$, we can see that $L_\cB$ is primitively embedded into $\Pic(X_2)$, and therefore $L_\cB=S(\sigma)$. Furthermore, from this explicit description and again using a similar computation as described earlier, we can compute that $S(\sigma)$ has discriminant form $\omega_{2,2}^5$. Since it is hyperbolic and of rank 13, by Theorem~\ref{unicity} we have $S(\sigma)=U\oplus E_8\oplus A_3$.

\subsection{The second K3 surface}
Let $X_4$ be a K3 surface with purely non-symplectic automorphism of order 16 such that $|\disc\Pic (X_4)|=2^4$. Again in this case, Brandhorst \cite{Br} has shown that this K3 surface is unique and that $\Aut(X_4)\cong \ZZ/2\ZZ\times \ZZ/16\ZZ$. Thus there are two purely non-symplectic automorphisms of order 16 on this K3 surface. These fit into the last two lines of Table~\ref{tab:thm5.1} (also \cite[Theorem 5.1]{order16}), as we will see. 

By \cite{Br, order16, dillies16}, this K3 surface has an 
elliptic fibration that is invariant under both automorphisms. 
This elliptic fibration has the Weierstrass equation
\begin{equation}\label{e:elliptX_4}
y^2=x^3+t^3(t^4-1)x=x(x^2+t^3(t^4-1)).
\end{equation}

It is a well known fact that such an elliptic K3 surface has a 2-torsion section given by $t\mapsto (x(t),y(t))=(0,0)$ (see e.g. \cite[Section III.4]{SilvermanTate}). 
The two automorphisms are described in \cite{dillies16}. We will denote them as $\sigma$ and $\sigma'$:
\begin{align*}
\sigma:(x,y,z)&\mapsto (\xi_{16}^6x, \xi_{16}^9y,\xi_{16}^4t)\\
\sigma':(x,y,z)&\mapsto (\xi_{16}^6\frac{y^2-x^3}{x^2}, \xi_{16}^9\frac{x^3y-y^3}{x^3},\xi_{16}^4t).
\end{align*}

As noted in \cite{dillies16}, these two automorphisms commute, and $\sigma^{-1}\sigma'$ is the symplectic involution given by the translation by the section of order two (see \cite{dillies16} for more details).

The elliptic fibration \eqref{e:elliptX_4} has a fiber of type $III^*$ (an extended $E_7$) over $t=0$ and five fibers of type $III$, one of which lies over $t=\infty$. The other four are permuted by both automorphisms. From the Shioda-Tate formula, we know that the group of sections has rank zero, and there are exactly two sections (the zero section and the torsion section) in this fibration (see \cite[Table 1, No. 387]{shimada_arxiv} and \cite{shimada}). 
This gives us the configuration of curves in Figure~\ref{fig:fibX4}. Dotted lines represent sections.
The automorphism $\sigma$ permutes the four $III$ fibers, and the automorphism $\sigma'$ permutes the same four fibers and reflects the whole configuration from top to bottom. We will describe the fixed locus of both automorphisms more completely in the next section. 

\begin{figure}[ht]
\caption{Fibration on $X_4$, the dotted lines represent sections}

\centering\begin{tikzpicture}[xscale=.8,yscale=.6, thick]

	\draw [thick] (-2,-4)--(6.5,-4);

	\draw[scale=0.5, domain=-2.5:4.5,smooth,variable=\x] plot ({\x*\x/12+12},{\x-1});
	\draw[scale=0.5, domain=-4.5:2.5,smooth,variable=\x] plot ({-\x*\x/12+12},{\x-1});

	\draw[scale=0.5, domain=-2.5:4.5,smooth,variable=\x] plot ({\x*\x/12+2},{\x-1});
	\draw[scale=0.5, domain=-4.5:2.5,smooth,variable=\x] plot ({-\x*\x/12+2},{\x-1});

	\draw[scale=0.5, domain=-2.5:4.5,smooth,variable=\x] plot ({\x*\x/12+4},{\x-1});
	\draw[scale=0.5, domain=-4.5:2.5,smooth,variable=\x] plot ({-\x*\x/12+4},{\x-1});

	\draw[scale=0.5, domain=-2.5:4.5,smooth,variable=\x] plot ({\x*\x/12+6},{\x-1});
	\draw[scale=0.5, domain=-4.5:2.5,smooth,variable=\x] plot ({-\x*\x/12+6},{\x-1});

	\draw[scale=0.5, domain=-2.5:4.5,smooth,variable=\x] plot ({\x*\x/12+8},{\x-1});
	\draw[scale=0.5, domain=-4.5:2.5,smooth,variable=\x] plot ({-\x*\x/12+8},{\x-1});

	\node [below] at (-1,-4){$0$};

	\draw (-0.2,-1.4)--(-0.2,0.3);

	\draw (0,-0.9)--(-1.5,-1.6);
	\draw (-1.5,-1.4)--(0,-2.1);
	\draw (0,-1.9)--(-1.3,-2.8);

	\draw (0,-0.5)--(-1.5,-0.5);

	\draw (0,-0.1)--(-1.5,0.6);
	\draw (0,1.2)--(-1.5,0.4);
	\draw (0,0.8)--(-1.3,1.7);

	\draw [dotted] (-1.8,1.5)--(6.7,1.5);
	\draw [dotted] (-1.8,-2.5)--(6.7,-2.5);

	\node [below] at (5.8,-4){$\infty$};

	\node [below] at (-1.5,-0.4){\footnotesize$\cE_2$};
	\node [below] at (0.1,-0.4){\footnotesize$\cE_1$};
	\node [above] at (-1,1.5){\footnotesize$\cE_5$};
	\node [above] at (-1,0.6){\footnotesize$\cE_4$};
	\node          at (-1,0){\footnotesize$\cE_3$};
	\node [below] at (-1,-2.5){\footnotesize$\cE_8$};
	\node [below] at (-1,-1.5){\footnotesize$\cE_7$};
	\node [below] at (-1,-0.6){\footnotesize$\cE_6$};

	\node [below] at (5.9,-1.8){\footnotesize$\cI_2$};
	\node [below] at (3.9,-1.8){\footnotesize$\cI_4$};
	\node [below] at (2.9,-1.8){\footnotesize$\cI_6$};
	\node [below] at (1.9,-1.8){\footnotesize$\cI_8$};
	\node [below] at (0.9,-1.8){\footnotesize$\cI_{10}$};
	\node [below] at (6.1,1.6){\footnotesize$\cI_1$};
	\node [below] at (4.1,1.6){\footnotesize$\cI_3$};
	\node [below] at (3.1,1.6){\footnotesize$\cI_5$};
	\node [below] at (2.1,1.6){\footnotesize$\cI_7$};
	\node [below] at (1.1,1.6){\footnotesize$\cI_{9}$};
	
	\node [above] at (5.5,1.5){\footnotesize$S_1$};
	\node [below] at (4.5,-2.5){\footnotesize$S_2$};
	
\end{tikzpicture}
\label{fig:fibX4}
\end{figure}

\subsubsection*{Alternate proof of uniqueness of $X_4$} As in Section~\ref{sec31}, one can alternatively prove the uniqueness of this K3 surface geometrically. 
Let $\sigma$ be any automorphism of order 16 acting on the K3 surface $X$ as in the last two lines of Table~\ref{tab:thm5.1}.

We consider the jacobian elliptic fibration with a fiber of type $III^*=\tilde{E_7}$ and five fibers of type $III=\tilde{A_1}$, 
which we know is $\sigma^8$--invariant, since $\sigma^8$ acts as the identity on the Picard group. This comes from the fact that the lattice $U\oplus E_7\oplus A_1^{\oplus 5}$ is a sublattice of $\Pic(X)=U(2)\oplus D_4\oplus E_8$.

By \cite{shimada, shimada_arxiv} the fibration admits a $2$-torsion section that we denote  by $S_2$. 
In this case $\sigma^8$ fixes 5 rational curves and a curve of genus 2, by a result of Nikulin (see \cite{nikulin2}).
The five rational curves are the fibers 
$\mathcal{E}_4$, $\mathcal{E}_1$, $\mathcal{E}_7$ and the sections $S_1$ and $S_2$. The curve of genus two is a 2-section of the fibration, since $\sigma^8$ fixes four points on the generic fiber (2 are already contained in $S_1$ and $S_2$). By using \cite[Lemma 5]{order4} we compute $f\cdot \sigma(f)\leq 3$ as in the previous case and by Lemma~\ref{fibr_inv} we are left with the case $f\cdot \sigma(f)\in\{0,2\}$. Again we will show that $f\cdot \sigma(f)=2$ is not possible, which therefore implies that $f\cdot\sigma(f)=0$. 

Assume $f\cdot \sigma(f)=2$. 
This means that each fiber $\sigma(F)$ 
is a 2-section of the fibration induced by $f$. 
We study how $\sigma(F)$ meets a singular fiber of type $III$ in the fibration $f$. 
Observe that $S_1$ and $S_2$ pass through a fixed point for $\sigma^8$ on each of the two components and the curve $C$ goes through the tangency point.
Hence the automorphism  $\sigma^8$ preserves each of the two components of each $III$ fiber. This means that for a fiber $F$ of $f$ the
elliptic curve $\sigma(F)$ which is preserved by $\sigma^8$ must meet only one of the two components of a $III$ fiber. 

Fix one of these $III$ fibers, and let $I_j$ denote the component intersecting $\sigma(F)$.
This means that this rational curve $I_j$ is a 2-section
of the fibration $\sigma(f)$. By Hurwitz formula the $2:1$ morphism to the basis of the fibration $\mathbb{P}^1$ contains two ramification points, and so $I_j$ meets two fibers of the fibration induced by $\sigma(f)$ in one point each, which is a tangency point. We call these two fibers $\sigma(F_1)$ and $\sigma(F_2)$, resp. Since $I_j$ and the two fibers are $\sigma^8$--invariant, then $I_j$ and $\sigma(F_1)$, (resp. $I_j$ and $\sigma(F_2)$) must meet in a fixed point of the involution $\sigma^8$. Observe that the point can not be the tangency point of the $III$ fiber. If this were the case then e.g. $\sigma(F_1)$ would be tangent to $I_j$ and this would imply that $\sigma(F_1)$ meets the $III$ fiber with multiplicity strictly bigger than 2. Recall, however, that $\sigma(F_1)$ is a 2-section of the fibration induced by $f$. So $I_j$, $\sigma(F_1)$ and $\sigma(F_2)$ must all meet in the intersection point with the section (either $S_1$ or $S_2$), but then we have found two distinct fibers of the same elliptic fibration meeting at the same point, which is not possible. 
This shows the following (with the previous notations):
\begin{prop}
The fibration induced by $U\oplus E_7\oplus A_1^{\oplus 5}$ is $\sigma$--invariant.
\end{prop}

We have now to compute the Weierstrass equation $y^2=x^3+A(t)x+B(t), \ t\in \PP^1$, to conclude the proof of the unicity of $X_4$. Observe that $\sigma^8$ acts as the identity on the section,
so it acts as the identity on the basis of the fibration, but the automorphism cannot act with order eight on the basis of the fibration.

In fact the fiber of type $III^*$ is necessarily on a fixed point on $\PP^1$ for the action, say over $t=0$, and one of the five $III$ fibers is over the other fixed point, say $t=\infty$. We still have four fibers $III$ which must be permuted by $\sigma$. So $\sigma$ acts at most with order four on the basis of the fibration. 

If $\sigma$ acted trivially on the basis of the fibration then each smooth elliptic fiber would admit an automorphism of order $16$ (not a translation, since $\sigma$ acts purely non--symplectically) which is not possible. A similar argument applies if $\sigma^2$ is the identity on the basis of the fibration. So $\sigma$ acts with order four on the basis of the fibration. 

After imposing the condition of having a $III^*$ fiber over $t=0$ and a $III$ fiber over $t=\infty$, by \cite[Table IV.3.1]{Mir}, we see that in the Weierstrass form of the fibration we can write
\[
A(t)=t^3 P(t),\,\, B(t)=t^5 Q(t)
\]
with $\deg P(t)=4$ and $\deg Q(t)=5$.

We are left with four fibers more of type $III$, this implies that $Q(t)$ must contain a factor of degree four to the power two. This is clearly impossible, so that $Q(t)=0$, which implies $B(t)=0$. On the other hand $\deg P(t)=4$ and the fibers over its zeros are the four fibers $III$. 
We may now assume that the action on the basis of the fibration is $t\mapsto i t$. After applying this transformation to $P(t)$ we want to get $i^l P(t)$ for some power $l$ of $i$. 

Putting all these facts together we can write:
$$
P(t)=a t^4+ b 
$$
where $a$ and $b$ are both non--zero constants. We can now write our elliptic fibration as 
$$
y^2=x^3+t^3(a t^4+ b)x.
$$
We use the automorphisms of $\PP^1$ and the transformation $(x,y)\mapsto (\lambda^2 x, \lambda^3 y)$ to get rid of the constants and to get finally the equation:
$$
y^2=x^3+t^3(t^4-1)x.
$$
This concludes the proof of the unicity of $X_4$.

\subsubsection{Invariant lattices}
In order to describe the invariant lattices, we first label the curves in the configuration as in Figure~\ref{fig:fibX4}. That is, in the fiber over 0 (the $III^*$), label the curves $\cE_1,\cE_2,\dots, \cE_8$, with $\cE_1$ being the central curve (where the three branches meet). Let $\cE_2$ be the single curve intersecting $\cE_1$; let $\cE_3$, $\cE_4$ and $\cE_5$ be the upper branch with $\cE_3$ intersecting $\cE_1$, and finally let $\cE_6$, $\cE_7$ and $\cE_8$ the lower branch with $\cE_6$ intersecting $\cE_1$. Label the two section $S_1$ and $S_2$, so that $S_1$ (the upper section) intersects $\cE_5$ and $S_2$ (the lower section) intersects $\cE_8$. Finally label the curves in the type $III$ fibers as $\cI_1,\cI_2,\dots, \cI_{10}$ so that $\cI_1\cup \cI_2$ lies over $\infty$, and the others pairs are labelled consecutively, i.e. $\cI_3\cup \cI_4$ is a fiber, $\cI_5\cup \cI_6$ is a fiber, etc. Furthermore, assume $S_1$ intersects $\cI_1,\cI_3, \cI_5, \cI_7,\cI_9$ and $S_2$ intersects $\cI_2,\cI_4,\cI_6,\cI_8,\cI_{10}$. 
As with the previous example, this configuration of curves generates $\Pic(X_4)$, as we will see. Since $\rk\Pic(X_4)=14$, six of the 20 curves must be redundant. In order to more easily describe the primitive embeddings of the invariant lattices into $\Pic(X_4)$, we will describe two different bases for $\Pic(X_4)$ chosen from these curves. As before, we will also use these curves to describe a minimal set of generators for each invariant lattice. We will first compute the invariant lattice for $\sigma$ and then for $\sigma'$.  

\subsubsection{Invariant lattice for $\sigma$:}\label{sss:invariant}
As shown in \cite{dillies16}, the automorphism $\sigma$ leaves invariant the curves in the fiber over 0 (fixing $\cE_1$ pointwisely), the fiber over $\infty$ and the two sections. Furthermore, $\sigma$ permutes the curves $\cI_4,\cI_6,\cI_8$ and $\cI_{10}$ and $\cI_3,\cI_5,\cI_7$ and $\cI_9$. From the description of the action of $\sigma$, we can see that $\sigma$ fixes 10 isolated points (3 in the fiber above $\infty$ and 7 in the fiber above 0) and 1 rational curve ($\cE_1$). This is the second line in Table~\ref{tab:thm5.1}. The invariant lattice $S(\sigma)$ was computed in \cite[Example 4.5]{CP}. We will give an alternate description here. 
 
The configuration of curves in Figure~\ref{fig:fibX4} can be represented by the diagram in Figure~\ref{fig:graphX_4}. The vertices depicted as double circles represent the sections, and the weighted edges labeled with 2 represent the tangent intersection of the two curves in the $III$ fibers. 

\begin{figure}[ht]
	\caption{Configuration of curves on $X_4$. The vertices with two circles represent the sections.}
	\centering
\begin{tikzpicture}[xscale=.6,yscale=.5, thick]
	
\begin{scope}[every node/.style={circle, draw, fill=black!50, inner sep=0pt, minimum width=4pt}]
	\draw (-1,0) node{} -- (0,0) node[name=e1]{}
	-- (0,1) node{}
	-- (1,2) node{}
	-- (2.5,2) node[name=e5]{};

	\draw (e1) -- (0,-1) node{}
	-- (1,-2) node{}
	-- (2.5,-2) node[name=e8]{};

	\draw	(e5) -- (4,2) node[name=s1, double, draw, inner sep=0pt, minimum width=4pt]{};
	\draw 	(e8) -- (4,-2) node[name=s2, double, draw, inner sep=0pt, minimum width=4pt]{};

	\draw (s1) -- (2,0.6) node{}
	-- (2,-0.6) node{}
	-- (s2);
	
	\draw (s1) -- (3,0.6) node{}
	-- (3,-0.6) node{}
	-- (s2);
	
	\draw (s1) -- (4,0.6) node{}
	-- (4,-0.6) node{}
	-- (s2);
	
	\draw (s1) -- (5,0.6) node{}
	-- (5,-0.6) node{}
	-- (s2);
	
	\draw (s1) -- (6,0.6) node{}
	-- (6,-0.6) node{}
	-- (s2);

\end{scope}
	
	\node[left] at (2,0) {2};
	\node[left] at (3,0) {2};
	\node[left] at (4,0) {2};
	\node[left] at (5,0) {2};
	\node[left] at (6,0) {2};
	
\end{tikzpicture}
	\label{fig:graphX_4}
\end{figure}
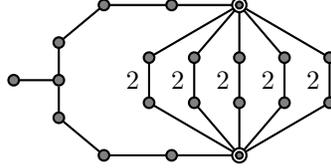

We need to find a suitable minimal set of generators for $\Pic (X_4)$. Again as before, we can study the matrix obtained from this graph, and one can check (via computer computation) that $\Pic (X_4)$ is generated by $\set{\cE_1,\dots, \cE_8, S_1,S_2, \cI_1, \cI_4, \cI_6, \cI_{8}}$. This is depicted in Figure~\ref{fig:PicGenX_4}.

\begin{figure}[ht]
	\caption{Configuration of curves on $X_4$. The filled vertices form a minimal set of generators of $\Pic(X_4)$}

	\centering
	\begin{tikzpicture}[xscale=.6,yscale=.5, thick]
	
	\begin{scope}[every node/.style={circle, draw, fill=black!50, inner sep=0pt, minimum width=4pt}]
	\draw (-1,0) node{} -- (0,0) node[name=e1]{}
	-- (0,1) node{}
	-- (1,2) node{}
	-- (2.5,2) node[name=e5]{}
	-- (4,2) node[name=s1]{};
	
	\draw (e1) -- (0,-1) node{}
	-- (1,-2) node{}
	-- (2.5,-2) node[name=e8]{}
	-- (4,-2) node[name=s2]{};

\end{scope}	

\begin{scope}[every node/.style={circle, fill=black!50, draw, inner sep=0pt, minimum width=4pt}]
	
	\draw (s1) -- (6,0.6) node[name=i1]{};
	
	\draw (3,-0.6) node[name=i4]{}
	-- (s2);
	
	\draw (4,-0.6) node[name=i6]{}
	-- (s2);
	
	\draw (5,-0.6) node[name=i8]{}
	-- (s2);
	
\end{scope}

	\node[circle, draw, inner sep=0pt, minimum width=4pt, name=i2] at (6,-0.6) {};
	\node[circle, draw, inner sep=0pt, minimum width=4pt, name=i3] at (3,0.6) {};
	\node[circle, draw, inner sep=0pt, minimum width=4pt, name=i5] at (4,0.6) {};
	\node[circle, draw, inner sep=0pt, minimum width=4pt, name=i7] at (5,0.6) {};
	\node[circle, draw, inner sep=0pt, minimum width=4pt, name=i9] at (2,0.6) {};
	\node[circle, draw, inner sep=0pt, minimum width=4pt, name=i10] at (2,-0.6) {};

	\draw[dotted] (i1) -- (i2) -- (s2);
	\draw[dotted] (s1) -- (i9) -- (i10) -- (s2);
	\draw[dotted] (s1) -- (i3) -- (i4);
	\draw[dotted] (s1) -- (i5) -- (i6);
	\draw[dotted] (s1) -- (i7) -- (i8);

	\end{tikzpicture}
	\label{fig:PicGenX_4}
\end{figure}
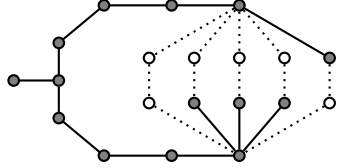

In order to find $S(\sigma)$, we construct a lattice $L_\cB$ generated by 
\[
\cB=\{\cE_1\dots, \cE_8, S_1,S_2, \cI_1,\cI_2, \cI_4+\cI_6+\cI_8+\cI_{10}, \cI_3+\cI_5+\cI_7+\cI_9\}
\] 
and, from the description of the action of $\sigma$ given above, we can see it is contained in $S(\sigma)$. The graph for this set is depicted in Figure~\ref{fig:X4invariantGraph}. The two square vertices represent the sums of $\cI_4+\cI_6+\cI_8+\cI_{10}$ and $\cI_3+\cI_5+\cI_7+\cI_9$, resp. and each has a self-intersection number $-8$. The edge weighted with 8 represents the intersection $(\cI_4+\cI_6+\cI_8+\cI_{10})\cdot (\cI_3+\cI_5+\cI_7+\cI_9)=8$.

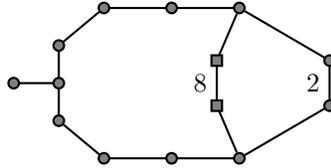
\begin{figure}[ht]
	\caption{Configuration of curves on $X_4$ generating $L_\cB$.}
	\centering
	\begin{tikzpicture}[xscale=.6,yscale=.5, thick]
	
	\begin{scope}[every node/.style={circle, draw, fill=black!50, inner sep=0pt, minimum width=4pt}]
	\draw (-1,0) node{} -- (0,0) node[name=e1]{}
	-- (0,1) node{}
	-- (1,2) node{}
	-- (2.5,2) node[name=e5]{}
	-- (4,2) node[name=s1]{};
	
	\draw (e1) -- (0,-1) node{}
	-- (1,-2) node{}
	-- (2.5,-2) node[name=e8]{}
	--(4,-2) node[name=s2]{};
	
	\draw (s1) -- (6,0.6) node{}
	-- (6,-0.6) node{}
	-- (s2);
	
	\end{scope}
	
	\draw (s1) -- (3.5,0.6) node[style=rectangle, draw, inner sep=0pt, minimum width=4pt,minimum height=4pt, fill=black!50]{}
	-- (3.5,-0.6) node[style=rectangle, draw, inner sep=0pt, minimum width=4pt,minimum height=4pt, fill=black!50]{}
	-- (s2);

	\node[left] at (3.5,0) {8};
	\node[left] at (6,0) {2};

	\end{tikzpicture}
	\label{fig:X4invariantGraph}
\end{figure}

From Table~\ref{tab:thm5.1}, we know that $\rk S(\sigma)=11$. Again using a similar computation to those described previously, we see that the last three generators in the set $\cB$ are redundant, and that $L_\cB$ is generated by $\{\cE_1,\dots, \cE_8, S_1,S_2, \cI_1\}$ as depicted in Figure~\ref{fig:X4invariantBasis}.
In particular $\rk L_\cB=\rk S(\sigma)$. 

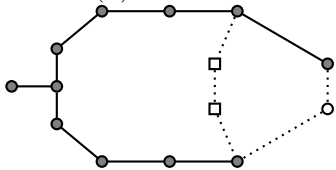
\begin{figure}[ht]
	\caption{Configuration of curves on $X_4$. The filled vertices represent a minimal set of generators of $S(\sigma)$.}

	\centering
	\begin{tikzpicture}[xscale=.6,yscale=.5, thick]
	
	\begin{scope}[every node/.style={circle, draw, fill=black!50, inner sep=0pt, minimum width=4pt}]
	\draw (-1,0) node{} -- (0,0) node[name=e1]{}
	-- (0,1) node{}
	-- (1,2) node{}
	-- (2.5,2) node[name=e5]{}
	-- (4,2) node[name=s1]{}
	-- (6,0.6) node{};
	
	\draw (e1) -- (0,-1) node{}
	-- (1,-2) node{}
	-- (2.5,-2) node[name=e8]{}
	-- (4,-2) node[name=s2]{};

	\end{scope}

	\node[circle, draw, inner sep=0pt, minimum width=4pt, name=i2] at (6,-0.6) {};
	\node[style=rectangle, draw, inner sep=0pt, minimum width=4pt,minimum height=4pt, name=iodd] at (3.5,0.6) {};
	\node[style=rectangle, draw, inner sep=0pt, minimum width=4pt,minimum height=4pt, name=ieven] at (3.5,-0.6) {};

	\draw[dotted] (i1) -- (i2) -- (s2);
	\draw[dotted] (s1) -- (iodd) -- (ieven) -- (s2);

	\end{tikzpicture}
	\label{fig:X4invariantBasis}
\end{figure}

The embedding $L_\cB\hookrightarrow \Pic (X_4)$ is clearly primitive; 
in fact we have shown that the set $\{\cE_1,\ldots,\cE_8, S_1,S_2, \cI_1,\cI_4,\cI_6,\cI_8\}$ 
is a set of generators of $\Pic(X_4)$ and $\{\cE_1,\ldots, \cE_8, S_1,S_2, \cI_1\}$
 is a set of generators of $L_\cB$. 
Since $\rk L_\cB=\rk S(\sigma)$, these two lattices must be equal. With this explicit description, we can compute the discriminant form of $S(\sigma)$ is $\omega_{2,3}^{-5}$, and therefore, by Theorem~\ref{unicity}, the lattice is $S(\sigma)\cong T_{2,5,6}$, which agrees with \cite{CP}. In fact we can recognize the $T_{2,5,6}$ in Figure~\ref{fig:X4invariantBasis}.

\subsubsection{Invariant lattice for $\sigma'$} 

As shown in \cite{dillies16}, the automorphism $\sigma'$ only fixes the curves $\cE_1$, $\cE_2$. It permutes transitively each of the following pairs of curves: $(\cE_3, \cE_6)$, $(\cE_4, \cE_7)$, $(\cE_5, \cE_8)$, $(S_1,S_2)$, $(\cI_1,\cI_2)$, and finally permutes the other curves in orbit of size four, namely $\cI_3, \cI_6, \cI_7, \cI_{10}$ are permuted cyclically and $\cI_4, \cI_5, \cI_8, \cI_9$ are permuted cyclically. This action is depicted in Figure~\ref{fig:X4sigma2action}. One can check that $\sigma'$ fixes only four isolated points, so this corresponds to the third line in Table~\ref{tab:thm5.1}.  

\begin{figure}[ht]
	\caption{Action of $\sigma'$ on the configuration of curves.}
\centering
\begin{tikzpicture}[xscale=.6,yscale=.5, thick]
	
\begin{scope}[every node/.style={circle, draw, fill=black!50, inner sep=0pt, minimum width=4pt}]
	\draw (-1,0) node{} -- (0,0) node[name=e1]{}
	-- (0,1) node{}
	-- (1,2) node{}
	-- (2.5,2) node[name=e5]{};

	\draw (e1) -- (0,-1) node{}
	-- (1,-2) node{}
	-- (2.5,-2) node[name=e8]{};

	\draw	(e5) -- (4,2) node[name=s1, draw, inner sep=0pt, minimum width=4pt]{};
	\draw 	(e8) -- (4,-2) node[name=s2, draw, inner sep=0pt, minimum width=4pt]{};

	\draw (s1) -- (2,0.6) node{}
	-- (2,-0.6) node{}
	-- (s2);
	
	\draw (s1) -- (3,0.6) node{}
	-- (3,-0.6) node{}
	-- (s2);
	
	\draw (s1) -- (4,0.6) node{}
	-- (4,-0.6) node{}
	-- (s2);
	
	\draw (s1) -- (5,0.6) node{}
	-- (5,-0.6) node{}
	-- (s2);
	
	\draw (s1) -- (6,0.6) node{}
	-- (6,-0.6) node{}
	-- (s2);
	
	\end{scope}

		\draw[<->,>=stealth,bend right] (-1.5,1) to (-1.5,-1);

	\draw[->,>=stealth] (2.1,0.5) -- (2.9,-0.5);
	\draw[->,>=stealth] (3.1,-0.5) -- (3.9,0.5);
	\draw[->,>=stealth] (4,0.5) -- (5,-0.5);
	
	\draw[->,>=stealth] (2.1,-0.5) -- (2.9,0.5);
	\draw[->,>=stealth] (3.1,0.5) -- (3.9,-0.5);
	\draw[->,>=stealth] (4,-0.5) -- (5,0.5);
	
	\end{tikzpicture}
	\label{fig:X4sigma2action}
\end{figure}
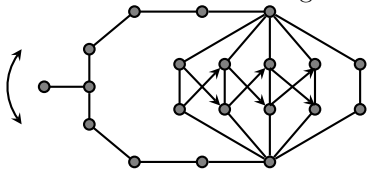

From the previous minimal set of generators for $\Pic (X_4)$, it would be very difficult to exhibit a primitive embedding, so in order to describe this invariant lattice, we need to find a new minimal set of generators for $\Pic(X_4)$ better suited to this automorphism. That new minimal set of generators is given by the 14 curves $\{\cE_1,\dots, \cE_8, S_1,S_2, \cI_3, \cI_7, \cI_6, \cI_{10}\}$, as one can see from the corresponding matrix. This minimal set of generators is represented in Figure~\ref{fig:PicGenX_42}. Notice that we have chosen a minimal set of generators compatible with the orbits of $\sigma'$. 

\begin{figure}[ht]
	\caption{Configuration of curves on $X_4$. The filled vertices form a minimal set of generators of $\Pic(X_4)$.}
	\centering
\begin{tikzpicture}[xscale=.6,yscale=.5, thick]
	
\begin{scope}[every node/.style={circle, draw, fill=black!50, inner sep=0pt, minimum width=4pt}]
	\draw (-1,0) node{} -- (0,0) node[name=e1]{}
	-- (0,1) node{}
	-- (1,2) node{}
	-- (2.5,2) node[name=e5]{}
	-- (4,2) node[name=s1]{};
	
	\draw (e1) -- (0,-1) node{}
	-- (1,-2) node{}
	-- (2.5,-2) node[name=e8]{}
	-- (4,-2) node[name=s2]{};

\end{scope}	
	
\begin{scope}[every node/.style={circle, fill=black!50, draw, inner sep=0pt, minimum width=4pt}]
	
	\draw (s1) -- (5,0.6) node[name=i3]{};
	
	\draw (s1) -- (3,0.6) node[name=i7]{};
	
	\draw (2,-0.6) node[name=i10]{}
	-- (s2);
	
	\draw (4,-0.6) node[name=i6]{}
	-- (s2);

\end{scope}
	
	\node[circle, draw, inner sep=0pt, minimum width=4pt, name=i2] at (6,-0.6) {};
	\node[circle, draw, inner sep=0pt, minimum width=4pt, name=i1] at (6,0.6) {};
	\node[circle, draw, inner sep=0pt, minimum width=4pt, name=i8] at (3,-0.6) {};
	\node[circle, draw, inner sep=0pt, minimum width=4pt, name=i5] at (4,0.6) {};
	\node[circle, draw, inner sep=0pt, minimum width=4pt, name=i4] at (5,-0.6) {};
	\node[circle, draw, inner sep=0pt, minimum width=4pt, name=i9] at (2,0.6) {};
	
	\draw[dotted] (s1) -- (i1) -- (i2) -- (s2);
	\draw[dotted] (s1) -- (i9) -- (i10);
	\draw[dotted] (i3) -- (i4) -- (s2);
	\draw[dotted] (s1) -- (i5) -- (i6);
	\draw[dotted] (i7) -- (i8) -- (s2);

\end{tikzpicture}
	\label{fig:PicGenX_42}
\end{figure}
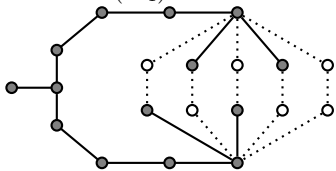

In order to compute the invariant lattice, we first consider the lattice generated by these orbits, namely the lattice $L_{\cB}$ generated by 
\[
\cB=\{\cE_1,\cE_2, \cE_3+\cE_6, \cE_4+\cE_7, \cE_5+\cE_8, S_1+S_2, \cI_1+\cI_2, \cI_3+ \cI_6+ \cI_7+ \cI_{10}, \cI_4+ \cI_5+ \cI_8+ \cI_9\}
\] 
giving us the configuration in Figure~\ref{fig:X4invariantGraph2}. The pentagon represents $\cI_1+\cI_2$ with self-intersection 0, the squares represent the other sums of two curves and each has self-intersection $-4$, and the triangles represent the sums of four curves and each has self-intersection $-8$.

\begin{figure}[ht]
	\caption{Configuration of generators of $L_\cB$.}
	\centering
\begin{tikzpicture}[xscale=.6,yscale=.5, thick]
	
	
\begin{scope}[every node/.style={circle, fill=black!50, draw, inner sep=0pt, minimum width=4pt}]
	
	\draw (-1,0) node{} -- (0,0) node[name=e1]{};
	
\end{scope}
	
\begin{scope}[every node/.style={style=rectangle, draw, inner sep=0pt, minimum width=4pt,minimum height=4pt, fill=black!50}]
	\draw (e1)
	-- (0,1) node{}
	-- (1,2) node{}
	-- (2.5,2) node[name=e5]{}
	-- (4,2) node[name=s1]{};
	
\end{scope}	
	
\begin{scope}[every node/.style={regular polygon,regular polygon sides=3, fill=black!50, draw, inner sep=0pt, minimum width=5pt}]
	
	\draw (s1) -- (3,0.6) node[name=i3]{};
	
	\draw (s1) -- (4,0.6) node[name=i5]{};
	
\end{scope}
	
	\draw (i3) -- (i5);
	\draw (s1) -- (6,0.6) node[regular polygon,regular polygon sides=5, fill=black!50, draw, inner sep=0pt, minimum width=5pt]{};
	
	\node[right] at (0,0.5) {2};
	\node[above] at (0.4,1.5) {2};
	\node[above] at (1.75,2) {2};
	\node[above] at (3.25,2) {2};
	\node[above] at (5,1.3) {2};
	\node[right] at (4,1.1) {4};
	\node[above] at (3.25,1) {4};
	\node[below] at (3.5,0.5) {8};
	
	\node[regular polygon,regular polygon sides=4, fill=black!50, draw, inner sep=0pt, minimum width=5pt ] at (7,2) {};
	\node[right] at (7,2) {=self-intersection $-4$};

	\node[regular polygon,regular polygon sides=3, fill=black!50, draw, inner sep=0pt, minimum width=5pt] at (7,1) {};
	\node[right] at (7,1) {=self-intersection $-8$};

	\node[regular polygon,regular polygon sides=5, fill=black!50, draw, inner sep=0pt, minimum width=5pt] at (7,0) {};
	\node[right] at (7,0) {=self-intersection $0$};

\end{tikzpicture}
	\label{fig:X4invariantGraph2}
\end{figure}
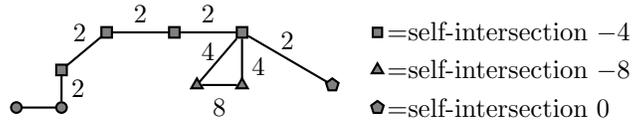 

Again after setting up the corresponding matrix, we can see that in fact the lattice $L_{\cB}$ has a minimal set of generators $\{\cE_1,\cE_2, \cE_3+\cE_6, \cE_4+\cE_7, \cE_5+\cE_8, S_1+S_2, \cI_3+ \cI_6+ \cI_7+ \cI_{10}\}$, as represented in Figure~\ref{fig:X_4invariantBasis2}. The minimal set of generators is represented by filled nodes, and the redundant sums are represented by hollow nodes.

\begin{figure}[ht]
	\caption{Configuration of generators of $L_{\cB}$. The filled vertices represent a minimal set of generators of $S(\sigma')$.}
	\centering
\begin{tikzpicture}[xscale=.6,yscale=.5, thick]
	
	
\begin{scope}[every node/.style={circle, fill=black!50, draw, inner sep=0pt, minimum width=4pt}]
	
	\draw (-1,0) node{} -- (0,0) node[name=e1]{};
	
\end{scope}
	
\begin{scope}[every node/.style={style=rectangle, draw, inner sep=0pt, minimum width=4pt,minimum height=4pt, fill=black!50}]
	\draw (e1)
	-- (0,1) node{}
	-- (1,2) node{}
	-- (2.5,2) node[name=e5]{}
	-- (4,2) node[name=s1]{};

\end{scope}	
	
\begin{scope}[every node/.style={regular polygon,regular polygon sides=3, fill=black!50, draw, inner sep=0pt, minimum width=5pt}]
	
	\draw (s1) -- (3,0.6) node[name=i3]{};
	
\end{scope}
	
	\node[regular polygon,regular polygon sides=3,  draw, inner sep=0pt, minimum width=5pt, name=i4] at (4,0.6) {};
	
	\node[regular polygon,regular polygon sides=5, draw, inner sep=0pt, minimum width=4pt, name=i1] at (6,0.6) {};
	
	\draw[dotted] (s1) -- (i1);
	\draw[dotted] (s1) -- (i4) -- (i3);
	
	\node[right] at (0,0.5) {2};
	\node[above] at (0.4,1.5) {2};
	\node[above] at (1.75,2) {2};
	\node[above] at (3.25,2) {2};
	\node[above] at (3.25,1) {4};
	
\end{tikzpicture}
	\label{fig:X_4invariantBasis2}
\end{figure} 

The embedding $L_{\cB}\hookrightarrow \Pic (X_4)$ is obviously primitive, so we know that $L_{\cB}=S(\sigma')$. We can compute the discriminant form is $u+v+\omega_{2,3}^{-1}$, which determines the lattice $U(2)\oplus D_4\oplus\langle -8\rangle$ by Theorem~\ref{unicity}.

\section{K3 surfaces from Table~\ref{tab:thm3.1}}\label{s:table2}

Finally, we consider a K3 surface $X_6$ with purely non-symplectic automorphism of order 16 such that $|\disc\Pic (X_6)|=2^6$. Such a surface must fall into one of two cases given by the two lines in Table~\ref{tab:thm3.1}. In \cite{Br}, Brandhorst has shown that such a K3 surface is unique
and he observes in \cite[Remark 7.3]{Br} 
that $\Aut(X_6)$ has infinite order. 
In fact this can be seen directly: we are going to describe an elliptic fibration on $X_6$ and, 
by using Shioda-Tate formula, we will see that it admits sections of infinite order.
Considering translation by these sections on each fiber, one gets symplectic automorphisms of infinite order on $X_6$. 

The K3 surface $X_6$ has an elliptic fibration given by the Weierstrass equation:
\begin{equation}
y^2=x^3+x+t^8. 
\label{fibration}\end{equation}

\begin{subsection}{The elliptic fibration}
We study here the properties of the elliptic fibration \eqref{fibration}.
The discriminant of the fibration is 
$\Delta(t)=4+27t^{16}$; 
it has $16$ simple zeros, so that the fibration has $16$ fibers of type $I_1$,
a fiber of type $IV^*$ (an extended $E_6$) over $t=\infty$ and a smooth fiber $\cC$ over $t=0$.

The fiber $\cC$ has equation
\[y^2=x^3+x=x(x-i)(x+i)\]
or in homogeneous coordinates
$zy^2=x^3+z^2x$
which is the equation of an elliptic curve that admits a complex multiplication of order four
$(x:y:z)\mapsto (-x: i y: z)$.

The fibration is jacobian so this means that it admits a section. 
We take the zero section $S_0$ as given by $z=0$, i.e.
\[
S_0:t\mapsto ((0:1:0),t). 
\]
Observe that the Picard rank of the fibration is 14. We will now express the Picard lattice in terms of root lattices.  

\begin{lem}\label{lem:PciY}
 The Picard group of $X_6$ is $\Pic (X_6)\cong U\oplus D_4^3$. 
\end{lem}

\begin{proof}
Recall from \cite{order16} that any K3 surface of this type admits an elliptic fibration with a $IV^*$ fiber over $\infty$ and a smooth curve over $0$. Furthermore, $\Pic(X_6)=S(\sigma^8)$, and $\sigma^8$ is a non-symplectic involution. This involution fixes the smooth fiber (with genus 1) over 0, and four of the curves making up the singular fiber over $\infty$. Looking at Nikulin's classification of non--symplectic involutions on K3 surfaces, $\Pic(X_6)$ has rank 14 and length 6. There are two lattices with these invariants, one with $\delta=0$ and one with $\delta=1$. We will show that $X_6$ has invariant $\delta=0$. The statement will then follow. 

Recall that $\delta=0$ is equivalent to say that the sum of the curves of the fixed locus of $\sigma^8$ is divisible by $2$ in the N\'eron--Severi group (see e.g. \cite[
Section 4.1]{ABS}).

Let $\cC$ denote the smooth elliptic curve over 0. We number the curves making up the $IV^*$ fiber over $\infty$: let $\cE_0$ denote the central curve, $\cE_1,\cE_2, \cE_3$ denote the three rational curves intersecting $\cE_0$ and let $\cE_4, \cE_5, \cE_6$ denote the last three rational curves. 

From the theory of elliptic fibrations, we know that 
\[
\cC=3\cE_0+2(\cE_1+\cE_2+\cE_3)+\cE_4+\cE_5+\cE_6.
\]

The sum of $\sigma^8$-invariant curves is 
\[
F=\cC+\cE_0+\cE_4+\cE_5+\cE_6.
\]

Combining the two facts, we have
\[
F=4\cE_0+2(\cE_1+\cE_2+\cE_3+\cE_4+\cE_5+\cE_6)
\]
which is divisible by 2.
\end{proof}

From Lemma~\ref{lem:PciY} we get in particular that the free part of the Mordell-Weil group of sections has rank 6,
i.e. it has a $\ZZ$-basis consisting of six sections of infinite order.
In fact by \cite{shimada} the fibration does not have torsion sections.

We now find some more interesting sections. If we set $x=0$ we get
\[
zy^2-z^3t^8=0
\]
and this factorizes as
\[
z(y^2-z^2t^8)=z(y-zt^4)(y+zt^4)=0.
\]

For $z=0$ we get again the zero section. 
Otherwise we find two more sections:
\[
S_1: t\mapsto ((0:t^4:1),t),\qquad S_2: t\mapsto ((0:-t^4:1),t).
\]
If we set $x=\pm i$ and $z\neq 0$ we have four more sections:
\[
S_3: t\mapsto ((i:t^4:1),t),\qquad S_4: t\mapsto ((i:-t^4:1),t).
\]\[
S_5: t\mapsto ((-i:t^4:1),t),\qquad S_6: t\mapsto ((-i:-t^4:1),t).
\]
Since the fibration does not admit torsion sections, these six sections have infinite order. 
Observe that $S_1$ and $S_2$ meet on the smooth fiber $\cC$ at the point $(0:0:1)$, whereas 
$S_3$ and $S_4$ meet at the point $(i:0:1)$
and 
$S_5$ and $S_6$ meet at the point $(-i:0:1)$. 
These three points are two-torsion points of the fiber $\cC$.

We need to see how these sections meet the singular fiber $IV^*$ of the elliptic fibration.
By the change of coordinates (see \cite[Section 3]{kondo_trivial}) 
\[
x_1=\frac x{t^4},\quad y_1=\frac y{t^6},\quad z=z_1,\quad t_1=1/t
\]
the equation of the elliptic fibration becomes
\[
z_1y_1^2=x_1^3+z_1^2x_1t_1^8+z_1^3t_1^4.
\]

For $t_1=0$ we get the equation of the fibration at infinity:
\[
z_1y_1^2=x_1^3,
\] 
and the fiber $IV^*$ comes from the blow up of the singular point $(0:0:1)$ in the cuspidal elliptic curve. 
The equations of the previous
sections $S_0,\ldots, S_6$ become
\begin{align*} 
&S_0:t_1\mapsto ((0:1:0), t_1),\\
&S_1:t_1\mapsto ((0:t_1^2:1),t_1),\\
&S_2:t_1\mapsto ((0:-t_1^2:1),t_1),\\
&S_3:t_1\mapsto ((it_1^4:t_1^2:1), t_1),\\ 
&S_4:t_1\mapsto ((it_1^4:-t_1^2:1), t_1), \\
&S_5: t_1\mapsto ((-it_1^4:t_1^2:1), t_1),\\
&S_6: t_1\mapsto ((-it_1^4:-t_1^2:1), t_1).
\end{align*}

Since the sections $S_1,\ldots, S_6$ meet each fiber with multiplicity one, 
they will meet the fiber $IV^*$ in an external component (these all have multiplicity one). 
Observe that if $t_1=0$ the six sections $S_1,\ldots, S_6$ meet at the singular point $(0:0:1)$, 
so that these will each meet a different component of the $IV^*$ fiber than the zero section, we make this more precise later.

\end{subsection}

\subsection{Automorphisms}

We know there are at least two purely non-symplectic automorphisms of order 16 on this surface, namely
\begin{align*}
\sigma:(x,y,z)&\mapsto (-x,iy,\xi_{16}^{13}t)\\
\sigma':(x,y,z)&\mapsto (-x,-iy,\xi_{16}^{5}t) 
\end{align*}
and clearly $\sigma^2=(\sigma')^2$. Thus for both automorphisms, the fourth power $\sigma^4$ fixes the fiber over $\infty$, as mentioned in Theorem~\ref{t:tables}. 
Observe that $\sigma^*$ and $(\sigma')^*$ each multiplies the 2-form by $\xi_{16}$. 

Before computing the invariant lattices we first give an alternate proof of the unicity of $X_6$. Our proof shows that in fact automorphisms of the same type as $\sigma$ and $\sigma'$ leave the elliptic fibration invariant. 

The group of automorphisms of $X_6$ is infinite, so we do not immediately see that the automorphism of order 16 is unique (recall Definition~\ref{def:unicite}). 
We prove this fact after we prove that $X_6$ is unique. In particular, we will see that any purely nonsymplectic automorphism of one of the types listed in Table~\ref{tab:thm3.1} is conjugate to either $\sigma$ or $\sigma'$.

\subsubsection*{Alternate proof of unicity of $X_6$}

We begin with the following lemma, detailing the fibers of an elliptic fibration. 

\begin{lem}\label{A1}
Let $Y$ be a K3 surface with purely non--symplectic automorphism $\varphi$ 
of order 16, and $\Pic(Y)=S(\varphi^8)$, and let $\cC$ be a genus 1 curve 
fixed by $\varphi^4$.
Then $Y$ admits a $\varphi$-invariant elliptic fibration, with 17 singular fibers, namely a fiber of type $IV^*$ over $\infty$, and 16 $I_1$ fibers, which are interchanged by $\varphi$.
\end{lem}

\begin{proof}
By \cite[Theorem 3.1]{order4} with these parameters, there exists a $\varphi$-invariant elliptic fibration, 
such that the fiber over $0$ is a smooth curve $\cC$ of genus 1, and the fiber over $\infty$ is of type $IV^*$. 
Moreover, $\varphi$ acts on the base of the fibration with order 16 (since $\varphi^4$ acts with order 4 on the base; see \cite[Proof of Theorem 3.1]{order16}).

Now consider the other possible fibers. First notice that if there are any other singular fibers, then $\varphi$ permutes them, and so there must be sixteen of each type. Furthermore, if there is a reducible singular fiber (and therefore 15 more) then $\rk \Pic(X)> 14$ which is impossible. 

Thus the singular fibers must be irreducible and of Euler characteristic 
\[
(24-8)/16 = 16 /16 = 1.
\] 
That narrows the possibilities to $I_1$ fibers or smooth fibers. If we denote by $e(\cdot)$ the Euler characteristic, and the fiber over $t\in \PP^1$ by $X_t$, then we have
\[
e(X)=\sum_{t\in \PP^1}e(X_t)=24.
\]
Since the Euler characteristic of a smooth fiber is 0, and $e(IV^*)=8$ (see \cite[Table IV.3.1]{Mir}), there must be at least one more singular fiber. So the sixteen $I_1$ fibers are the only possibility. 
\end{proof}

The next Lemma shows that the elliptic fibration mentioned above admits a section. Once we have a section, we can write down the Weierstrass form of the elliptic fibration $\mathcal{F}$.

\begin{lem}\label{A2}
The $\varphi$-invariant elliptic fibration $\mathcal F$ of Lemma~\ref{A1} has a section. 
\end{lem}

\begin{proof} 

Because $\Pic (Y)\cong U\oplus D_4^3$ 
by \cite[Lemma 2.1]{kondo_trivial} or \cite[\S 3, proof of Corollary 3]{shapiro} 
there is an elliptic fibration $\mathcal E$ induced by the inclusion $U\subseteq \Pic(Y)$ with 3 singular fibers of type $I_0^*$.
This is illustrated in Figure~\ref{fig:2nd_fibration}. The line labelled $S$ represents a section and the others make up the three singular fibers. 

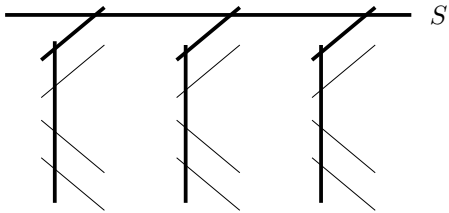
\begin{figure}[ht]\label{fig:Efibration}
\caption{The fibration $\mathcal E$.  The line labelled $S$ is the section for the fibration $\cE$. The thicker lines represent the singular fiber $IV^*$ for the fibration $\mathcal{F}$.}
	\centering

\begin{tikzpicture}[xscale=.6,yscale=.5]

	\draw [line width=0.5mm](-1,7)--(8,7);
	\node [right] at (8.2,7){$S$};

	\draw [line width=0.5mm](-0.2,5.8)--(1.2,7.2);
	\draw [line width=0.05mm](-0.2,4.8)--(1.2,6.2);
	\draw [line width=0.5mm](0.1,6.3)--(0.1,2);
	\draw [line width=0.05mm](1.2,2.8)--(-0.2,4.2);
	\draw [line width=0.05mm](1.2,1.8)--(-0.2,3.2);

	\draw [line width=0.5mm](5.8,5.8)--(7.2,7.2);
	\draw [line width=0.05mm](5.8,4.8)--(7.2,6.2);
	\draw [line width=0.5mm](6,6.2)--(6,2);
	\draw [line width=0.05mm](7.2,2.8)--(5.8,4.2);
	\draw [line width=0.05mm](7.2,1.8)--(5.8,3.2);

	\draw [line width=0.5mm](2.8,5.8)--(4.2,7.2);
	\draw [line width=0.05mm](2.8,4.8)--(4.2,6.2);
	\draw [line width=0.5mm](3,6.2)--(3,2);
	\draw [line width=0.05mm](4.2,2.8)--(2.8,4.2);
	\draw [line width=0.05mm](4.2,1.8)--(2.8,3.2);

\end{tikzpicture}
\label{fig:2nd_fibration}
\end{figure}

Since $\mathcal F$ is another fibration on the same surface, we should recognize its fibers in Figure~\ref{fig:2nd_fibration} too.
In Figure~\ref{fig:2nd_fibration} we show a fiber of type $IV^*$ made up of the thicker lines. 

\textit{A priori}, there might be another way of embedding the fiber $IV^*$ in $\Pic(Y)$. We will show this is not the case. 

In the elliptic fibration in Figure~\ref{fig:2nd_fibration}, the Mordell-Weil group has rank 0, and one can see that the fibration has only one section. Furthermore, because $\varphi^8$ fixes the Picard lattice, the elliptic fibration is $\varphi^8$-invariant.

We know that the invariant lattice for $\varphi^8$ has rank 14 and $a=6$, so from the classification of non-symplectic involutions it has a fixed locus of one genus 1 curve and 4 rational curves. 

We will find these fixed curves in the configuration in Figure~\ref{fig:2nd_fibration}. First we look at the section. If the section is not fixed pointwisely, then it is rotated by order 2. This means that two of the $I_0^*$ fibers would be permuted, and therefore then the Picard lattice is not invariant. So that leaves us with the section being fixed pointwisely.

From there, we understand exactly the action of $\varphi^8$ on the other curves in this configuration, namely, $\varphi^8$ fixes the section and the central curve of each of the singular fibers. It acts by order two on the curve connecting the section to the fixed curve of each fiber. And it must also act by order two on each of the other curves in each fiber. Now there can’t be any isolated fixed points, so the other fixed point on each of these rational curves must be a point of the genus one curve that is fixed. That means the genus one curve must be a 3-section. 

Now that we see how the curves are situated with regard to $\varphi^8$, we can then ask how does the $IV^*$ fiber fit into the picture from our other elliptic fibration. We know that the central curve of the $IV^*$ fiber as well as the extremal curves are fixed pointwisely by $\varphi^8$. There is only one way this configuration can fit into the fibration in Figure~\ref{fig:2nd_fibration}, so it has to be situated as we have described. 

The other curves not included in the fiber must therefore be sections and thus $\mathcal F$ admits a section.
\end{proof}

Now that we know there is a section, we can determine the Weierstrass equation $y^2=x^3+A(t)x+B(t), t\in\mathbb P^1$ of $\mathcal F$.

As before, let $a(t),b(t),\delta(t)$ denote the vanishing order of $A,B,\Delta$, resp., at $t\in \PP^1$.
By \cite[Table IV.3.1]{Mir}, one has $a(\infty)\geq 3, b(\infty)=4, \delta(\infty)=8$. Furthermore, $\Delta$ has 16 more simple zeroes $t_1,\ldots, t_{16}$ such that $a(t_i)=b(t_i)=0$.

Since $\deg A(t)=8$ and $a(\infty)\geq 3$, the degree of $A(t)$ is less or equal than 5 in the variable $t$. Since $\sigma$ acts on the base $\PP^1$ as an automorphism of order 16, if there were a root $\alpha$ of $A$ other than $0$ or $\infty$, there should be also the 15 other roots, obtained as the images of $\alpha$ under the group action by $\sigma$. This is impossible, thus $A(t)=1$, i.e. $a(\infty)=8$.

Since $\deg B(t)=12$ and $b(\infty)=4$, the degree of $B(t)$ in the variable $t$ is 8. 
Following \cite[Proof of Proposition 4.1]{order32}, we find that $B(t)=t^8$ and thus the Weierstrass equation of $\mathcal F$ is 
\[
y^2=x^3+x+t^8.
\]

By Lemmas~\ref{A1} and \ref{A2} the $\varphi$-invariant elliptic fibration has equation $y^2=x^3+x+t^8$ so that this admits the automorphisms $\sigma$, $\sigma'$ and $\varphi$. Observe that by the Shioda-Tate formula \cite[Corollary 1.5]{Shioda} and the classification of Shimada (see \cite{shimada} and  \cite[Table 1, No. 461]{shimada_arxiv}), the fibration has no torsion sections and its Mordell-Weil group is $\ZZ^6$.  

\subsubsection*{Unicity of $\sigma$ and $\sigma'$}

We will now show that $\sigma$ and $\sigma'$ are unique up to conjugation in $\Aut(X_6)$. We first need to know that the elliptic fibration we just described is invariant under the automorphism group $\Aut(X_6)$. On the way, we will be able to give a complete description of $\Aut(X_6)$.

\begin{lem}
	The elliptic fibration of Lemma~\ref{A1} is invariant under the automorphism group $\Aut(X_6)$. 
\end{lem}

\begin{proof}
Let $i=\sigma^8$; this is a non--symplectic involution.  Recall that we have an injective homomorphism, see \cite[Lemma 1.6]{MO}:
\[
\Aut(X_6) \to O(\Pic(X_6)) \times \GL(H^0(\Omega_{X_6}^2)),\,\,\sigma\mapsto (\sigma_{|\Pic(X_6)},\alpha(\sigma))
\] 
where the map to $\GL(H^0(\Omega_{X_6}^2)$ is defined by $\sigma^* \omega_{X_6}=\alpha(\sigma) \omega_{X_6}$.
Furthermore, $i$ acts trivially on $\Pic(X_6)$ (since $S(\sigma^8)=\Pic(X_6)$), by combining these facts 
one easily see that $i$ lies in the center of $\Aut(X_6)$.

The invariant lattice $S(\sigma^8)$ has rank 14 and length 6, so by the theory of Nikulin (see \cite{nikulin2}) 
our surface $X_6$ contains a unique smooth curve of genus 1 which is pointwise fixed by $i$. This is the curve $C$ from Lemma~\ref{A1}. Since $i$ is in the center of $\Aut(X_6)$, any automorphism must preserve $C$ (not necessarily pointwise).
Thus the complete linear system $|C|$ is $\Aut(X_6)$-invariant and therefore so is the elliptic fibration from Lemma~\ref{A1}. 

\end{proof}

Now we know that every automorphism of $X_6$ leaves the fibration invariant. Let $H\subseteq \Aut(X_6)$ denote the subgroup that leaves the zero section invariant. Notice $\sigma, \sigma'\in H$. In fact, we will show that any purely non-symplectic automorphism of order 16 is conjugate in $\Aut(X_6)$ to a power of $\sigma$ or $\sigma'$.

As mentioned previously, the Mordell-Weil group of sections of $X_6$ is isomorphic to $\ZZ^6$. Furthermore, each section of the elliptic fibration in question induces an automorphism of $X_6$. To set notation, let $\MW(X_6)$ denote the Mordell-Weil group. Given a section $P\in\MW(X_6)$, let us denote the corresponding automorphism by $T_P$, and addition (resp. subtraction) in $\MW(X_6)$ by $\oplus$ (resp. $\ominus$).  

In \cite[Section~3]{Karayayla}, we see that if $g\in H$, then we can act on $\MW(X_6)$ by conjugation as follows:
\begin{equation}\label{eq:MWAction}
gT_Pg^{-1}=T_{g(P)}. 
\end{equation}
Where $g(P)$ is the image of a section under $g$. Furthermore, we obtain the following theorem. This is essentially \cite[Theorem 3.0.1]{Karayayla}. (This is proved for relatively minimal rational elliptic surfaces, but the same proof also holds in our situation.) For completeness, we also include the proof here. 

\begin{lem}[cf. \cite{Karayayla}]\label{lem:semidirect}
	We have $\Aut(X_6)\cong \MW(X_6)\rtimes H$. 
\end{lem}

\begin{proof} 
	Let $g\in \Aut(X_6)$. Since $g$ preserves the fibration, it maps sections to sections. Let $P=g(O)$, and we see that $h=(T_{\ominus P})g$ fixes the zero section. Thus $h\in H$. So we see that $g=T_Ph$ with $T_P\in \MW(X_6)$ and $h\in H$. Furthermore, $\MW(X_6)\cap H=1$, which implies the result. 
\end{proof}

\begin{rem}
	Using Lemma~\ref{lem:semidirect} and \eqref{eq:MWAction}, we see that $\MW(X_6)$ is a normal subgroup of $\Aut(X_6)$. 
\end{rem}

\begin{lem}\label{lem:Hstructure}
	$H\cong\ZZ/2\ZZ\times \ZZ/16\ZZ$. 
\end{lem}

\begin{proof}
We have seen that every automorphism of $X_6$ leaves the fibration invariant. Moreover since the group $H$ 
leaves the zero section invariant, we can find a big and nef divisor in $\Pic(X_6)$ which is invariant by $H$, hence
$H$ is finite. Consider now the exact sequence of groups
$$
1\longrightarrow H_s\longrightarrow H\longrightarrow \mu_m\longrightarrow 1
$$ 
where the last map is $\alpha : H\longrightarrow \mu_m$, $g\mapsto \alpha(g)$ defined by $g^*\omega_{X_6}=\alpha(g)\omega_{X_6}$. 
In particular $H_s$ is the normal subgroup in $H$ of symplectic automorphisms and $\mu_m=H/H_s$ is a cyclic group
of some order $m$ (see \cite[Section 3]{nikulin3}). We know already that $H$ contains two non--symplectic automorphisms acting with a primitive $16$th--root of unity on $\omega_{X_6}$ hence $m$ is a multiple of $16$. Now recall that $\rk T_X=8$
and so $\phi(m)\leq 8$ which implies that $m=16$. 

We now need to find the structure of $H_s$. An element $\tau\in H_s$ acts on the base $\PP^1$ of the fibration either fixing each point or fixing two isolated points. In either case it must fix the point over which the fiber $IV^*$ lies, since there is only one fiber of this kind.
Furthermore, $\tau$ leaves the zero section invariant, so it acts on the $IV^*$ fiber either preserving each component or as a reflection. In the first case the central component of the $IV^*$ fiber would be pointwisely fixed. As a symplectic automorphism has only isolated fixed points, this is not possible (see \cite[Section 5]{nikulin3}). 

Hence $\tau$ acts on the $IV^*$ fiber as a reflection with four isolated fixed points. This means that the order of $\tau$ (on all of $X_6$) is even and it admits at least four fixed points. Observe that if the order of $\tau$ were bigger than $2$ then $\tau^2$ would be a non--trivial symplectic automorphism of $X_6$ of finite order with a curve in the fixed locus (the central component of the $IV^*$ fiber) which again is not possible. So $\tau$ has order $2$. 

Further observe that $\tau$ preserves a smooth fiber of the fibration, which must contains the other four isolated fixed points (recall a symplectic involution fixes exactly eight isolated points). In particular, $\tau$ cannot act as the identity on the base of the fibration, otherwise the nodes in the sixteen $I_1$ fibers would be fixed. 

So we see that any element of $H_s$ is an involution, that fixes the same four fixed points of the $IV^*$ fiber. 

If $\tau'\in H_s$ is another symplectic automorphism, then $\tau'\circ \tau$ is a finite order symplectic automorphism (in fact an involution, since it lies in $H_s$) that acts on $IV^*$ preserving each rational curve and so it fixes the 
central component of the fiber by the previous argument. This is impossible for a finite order symplectic automorphism. So $|H_s|=2$ which gives $|H|=2\cdot 16=32$. Since $H$ already contains the subgroup $\mathbb{Z}/2\mathbb{Z}\times \mathbb{Z}/16\mathbb{Z}$ generated by $\sigma$ and $\sigma'$, we are done. 
\end{proof}

Now we understand the structure of the automorphism group, we want to look more closely at automorphisms of order 16. 

\begin{lem}\label{lem:semidirect}
	Let $g\in\Aut(X_6)$ be an automorphism of order 16. Then $g=T_Ph$ with $h\in H$ satisfying $h^{16}=1$ and $P+h(P)+\dots + h^{15}(P)=O$.  
\end{lem}

\begin{proof}
	From what we have seen previously, we know that $g=T_Ph$ for some $h\in H$. A calculation shows that 
	\[
	g^{16}=T_{P+h(P)+\dots + h^{15}(P)}h^{16}.
	\]
	Since $H\cap \MW(X_6)=1$, the result follows. 
\end{proof}

Finally in order to prove unicity of $\sigma$ and $\sigma'$, we need to show that any other automorphism of order 16 is conjugate to a power of $\sigma$ or $\sigma'$ via an element of $\Aut(X_6)$. Recall that we have fixed a primitive 16th root of unity (see Definition~\ref{def:unicite}), and the invariant lattice and the fixed locus of such an automorphism do not depend on the choice of this root of unity. So the question becomes how can we recognize a conjugate of an element of the subgroup $H$? 

Given $h\in H$, we first consider the form of a conjugate of $h$:
\begin{align*}
f&=T_P h T_P^{-1}\\
  &= T_P (h T_{-P} h^{-1})h\\
  &= T_{P-h(P)}h\\
\end{align*}

Furthermore, from the previous lemma, we know that $f$ has order 16 if and only if $f=T_Qh$ with $Q\in \ker(1+h+\dots h^{15})$ and $h^{16}=1$. With these two facts, the last ingredient is the following lemma:

\begin{lem}\label{lem:endo}
	Let $h\in H$ be purely nonsymplectic with order 16. If we consider $1-h$ as an element of $\End(\MW(X_6))$, then $\Im(1-h)=\ker(1+h+\dots +h^{15})$. 
\end{lem}

\begin{proof}
First notice that $h^{16}-1=(h^{15}+\dots +h+1)(h-1)$, so 
\[
\Im(h-1)\subseteq \ker(h^{15}+\dots +h+1).
\] 
		
Furthermore, we can write $(x^{15} +\dots +1)-16 = (x-1)q(x)$ for some $q(x)\in \ZZ[x]$ using division. Thus if $P\in \ker(h^{15} +\dots +1)$, we have 
\[
16P=(1-h)q(P).
\] 

This shows us that $\rk \Im(1-h)= \rk\ker(1+h+\dots+h^{15})$, so $\Im(h-1)$ is of finite index in $\ker(h^{15} +\dots +1)$. In fact, this shows that the index is a divisor of 16. We need to see that this index is 1.
	
Furthermore, we can see that $\rk\ker(1-h)\geq 1$. For example, we will show that the sections $S_1$ and $S_2$ (see Figure~\ref{fig:sections}) are fixed by $\sigma$, and the section $S_1+S_2$ is fixed by $\sigma'$. Since we are working with free $\ZZ$-modules, we see that 
\[
\rk\Im(1-h)=
6-\rk\ker(1-h)\leq 5.
\] 
In order to show  that $[\Im(1-h): \ker(1+h+\dots+h^{15})]=1$, we will tensor with $\FF_2$, the field of order 2, and show that $(1-h)\otimes \FF_2\in \End(\FF_2^6)$ has rank 5. This will then imply the result. 
	
	
First, we consider a little more lattice theory. We will show that the action of $h$ on $\End(\FF_2^6)$ is the same as the action of $h$ on $T^\vee_X/T_X$. From there we will be able to explicitly describe the action to see it is rank 5.

To begin, let $N=U\oplus E_6$ generated by the $IV^*$ fiber and the zero section, and let $M=N^\perp$ the orthogonal complement in $\Pic(X_6)$ (these are the so-called {\it Trivial lattice} and $N^\perp$ is the opposite of the {\it Narrow Mordell-Weil lattice}, resp.; see \cite[Chapter 6]{SchSh}). Tensoring with the 3-adics $\ZZ_3$, we see that $\Pic(X_6)\otimes \ZZ_3$ is unimodular, so 
\[
(q_M)_3=-(q_N)_3=-q_N,
\]
since $N$ is 3--elementary. Here the subscript indicates we consider only the $\ZZ/3\ZZ$ part of the (finite) discriminant quadratic form (see e.g. \cite{nikulin}). 
 	
Furthermore, if we tensor with the 2-adics $\ZZ_2$, we see that $N\otimes\ZZ_2$ is unimodular, so 
\[
\Pic(X_6)\otimes \ZZ_2=(N\otimes \ZZ_2)\oplus(M\otimes \ZZ_2).
\] 
From these two facts, we obtain 
\[
q_M=-q_N\oplus q_{X_6},
\]  	
which together with the signature of $M$ determines the genus (see \cite[Corollary 1.9.4]{nikulin}).

Notice further that $\rk M=l(A_M)_2=6$, and since $(q_M)_2$ takes values in the 2-adics, we can divide the bilinear form on $M$ by 2 (see e.g. \cite[Proof of Theorem 1.16.4]{nikulin}) to obtain an even, negative definite lattice of rank 6 and determinant 3. The only such lattice is $E_6$, so we obtain $M=E_6(2)$. In fact by \cite[Theorem 6.44]{SchSh}, the lattice $M=E_6(2)$ is a sublattice of $\MW(X_6)$ of index 3.
So that over the $2$-adics, we see $E_6(2)\otimes \ZZ_2=\MW(X_6)\otimes \ZZ_2$. On the other hand over the $3$-adics $\Pic(X_6)$ is unimodular so that by \cite[Theorem 6.51]{SchSh} we know that on the $3$-adics the Mordell-Weil lattice $\MW(X_6)$ is isomorphic the dual of the narrow Mordell-Weil lattice which is $(E_6(-2))^\vee$,
by combining these two facts we get that $\MW(X_6)=E_6^\vee (-2)$, where we take $\MW(X_6)$ as a lattice with the {\it height pairing} (for the defintion of height pairing see \cite[Definition 6.21]{SchSh} and observe that $\MW(X_6)$ does not need to be an integral lattice \cite[Remark 6.25]{SchSh}). By using the decomposition:
\[
\Pic(X_6)\otimes \ZZ_2=(N\otimes \ZZ_2)\oplus(M\otimes \ZZ_2).
\] 
and the fact that the K3 lattice is unimodular we get that 
\[
\MW(X_6)\otimes \FF_2\cong \Pic(X_6)^\vee/\Pic(X_6)\cong T_X^\vee/T_X,
\] 
and this isomorphism is $h$--equivariant.

	The last step is to understand the action of $h$ on $T^\vee_X/T_X$. By \cite{Br}, we know $T_X^\vee/T_X=\ZZ[\zeta_{16}]/I$ where $I$ is the unique ideal in $\ZZ[\zeta_{16}]$ with norm $2^6$. This ideal is $(1-\xi_{16})^6$. 
	 
	On $\ZZ[\zeta_{16}]/(1-\zeta_{16})^6$ the automorphism $h$ acts by sending $x$ to $\zeta_{16} x$, so the minimal polynomial is the generator of the kernel of 
	\[
	\FF_2[x]\to \ZZ[\xi_{16}]/(1-\xi_{16})^6. 
	\]
	The codomain of this map is isomorphic to $\FF_2[x]/(1-x)^6$, since $(1-x)^6=\gcd(x^8-1, 1-x)$,
	so the minimal polynomial for the action is $(x-1)^6$. 
	This means that $h\otimes \FF_2$ has a single Jordan block of size 6 with eigenvalue 1. In particular the rank of $(h - 1)\otimes \FF_2$ is 5. 	
\end{proof}

Finally, we can prove the unicity of $\sigma$ and $\sigma'$. 

\begin{thm}
	If $g\in \Aut(X_6)$ is a purely non-symplectic automorphism of order 16, then $g$ is conjugate to some power of $\sigma$ or $\sigma'$. 
\end{thm}

\begin{proof}
If $g\in \Aut(X_6)$ is of order 16, then from Lemma~\ref{lem:semidirect}, we see that $g=T_Q h$ with $Q\in \ker(1+h+\dots+h^{15})$, and $h\in H$. By Lemma~\ref{lem:endo}, we have $Q=(1-h)(P)$ for some $P\in\MW(X_6)$, so we have
\begin{align*}
g&=T_Qh\\
 &=T_{P-h(P)}h\\
 &=T_P h T_p^{-1}.
\end{align*}

Finally, since $g$ is purely non-symplectic, and $T_P$ is symplectic, it follows that $h$ must also be purely non-symplectic. 
By Lemma~\ref{lem:Hstructure}, powers of $\sigma$ and $\sigma'$ are the only purely non-symplectic elements of $H$ of order 16.
\end{proof}

We will now study more in details the actions of $\sigma$ and $\sigma'$.

{\bf Action of $\sigma$}.
The automorphism $\sigma$ acts with order four on the fiber over $t=0$ and it fixes the two points $(0:1:0)$ and $(0:0:1)$; 
moreover, by \cite{order16}, it fixes $1$ rational curve, which is the central component of the $IV^*$ fiber, and $8$ isolated points: two on the curve $\cC$ and $6$ on the $IV^*$ fiber. 

By taking the coordinates at infinity we find that the action of $\sigma$ is as follows
\[
(x_1,y_1,t_1)\mapsto (ix_1,\xi_{16}^6 y_1, \xi_{16}^3 t_1).
\]

Clearly on the singular fiber
\[
z_1y_1^2=x_1^3
\]
it fixes the point $(0:1:0)$ and the singular point $(0:0:1)$.

The automorphism $\sigma$ leaves the zero section invariant fixing the points corresponding to the fibers $t=0$ and $t=\infty$. 
We look at the action on the other sections: $S_1$ and $S_2$ are preserved since

\[
((0:\pm t^4:1),t)\mapsto ((0:\pm it^4:1),\xi_{16}^{13}t)=((0:\pm (\xi_{16}^{13}t)^4:1),\xi_{16}^{13}t);
\]
and $\sigma$ fixes two points on each. 

On the other hand
\[
((\pm i:\pm t^4:1),t)\mapsto ((\mp i:\pm it^4:1),\xi_{16}^{13}t)=((\mp i:\pm(\xi_{16}^{13}t)^4:1),\xi_{16}^{13}t)
\] 
thus $S_3$ and $S_5$ are exhanged by $\sigma$ and the same is true for $S_4$ and $S_6$. 
Observe $\sigma^2$ preserves  all the sections $S_1,\ldots,S_6$ and by \cite{order16} it has an isolated fixed point on each of the external components of the $IV^*$ fiber so that $S_2, S_4$ and $S_6$ meet in the same point and the same holds for $S_1,S_3$ and $S_5$ as shown in Figure~\ref{fig:sections}.

Finally observe that on the smooth fiber $\cC$ the automorphism $\sigma$ exchanges the two order two points $(i:0:1)$ and $(-i:0:1)$.

{\bf Action of $\sigma'$}.
The automorphism $\sigma'$ acts with order four on the fiber over $t=0$ and it fixes the two points $(0:1:0)$ and $(0:0:1)$; moreover by \cite{order16}, it acts as a reflection on the $IV^*$ fiber and it fixes $6$ isolated points: two on the curve $\cC$ and $4$ on the $IV^*$ fiber. 

By taking the coordinates at infinity we find that the action of $\sigma'$ is as follows
\[
(x_1,y_1,t_1)\mapsto(ix_1,-\xi_{16}^6 y_1, \xi_{16}^{11} t_1).
\]

Then, as with $\sigma$, it fixes the point $(0:1:0)$ and the singular point $(0:0:1)$.
Moreover $\sigma'$ preserves the zero section and fixes on it the points corresponding to the fibers $t=0$ and $t=\infty$. 
By a similar computation, $\sigma'$ exchanges pairwise the sections $S_1$ and $S_2$, $S_3$ and $S_6$, and $S_4$ and $S_5$. 
This forces the two sections $S_1$ and $S_2$ to meet each an external component of the singular fiber $IV^*$ that the zero section does not meet ($S_1$ and $S_2$ do not meet the zero section). 
The same must be true for the pair $S_3$ and $S_6$ and for the pair $S_4$ and $S_5$.
Finally $\sigma'$ exchanges the two order two points $(i:0:1)$ and $(-i:0:1)$ on $\cC$, as was the case with $\sigma$ as well. 

Therefore the diagram is as in Figure~\ref{fig:sections}.

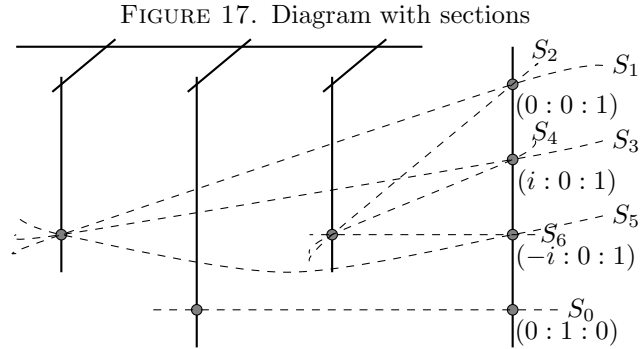
\begin{figure}[ht!]
\caption{Diagram with sections}
\centering\begin{tikzpicture}[xscale=.6,yscale=.5]

\draw [thick](-1,7)--(8,7);

\draw [thick](-0.2,5.8)--(1.2,7.2);
\draw [thick](0,6.2)--(0,1);

\draw [thick](5.8,5.8)--(7.2,7.2);
\draw [thick](6,6.2)--(6,1);

\draw [thick](2.8,5.8)--(4.2,7.2);
\draw [thick](3,6.2)--(3,-1);

\draw [thick](10,7)--(10,-1);
\begin{scope}[every node/.style={circle, draw, fill=black!50, inner sep=0pt, minimum width=4pt}]

\draw (10,6) node{};
\draw (10,4) node{};
\draw (10,2) node{};
\draw (10,0) node{};
\draw (3,0) node{};
\draw (0,2) node{};
\draw (6,2) node{};
\end{scope}

\draw [ dashed] (11,0)--(2,0);
\node[ right] at (11,0) {$S_0$};

\draw   [dashed] plot [smooth] coordinates {(10.5,6.5) (10,6) (6,2) (5.5,1.3)};
\node[right] at (10.2,6.9) {$S_2$};

\node[right] at (10.2,4.7) {$S_4$};
\draw  [dashed] plot [smooth] coordinates {(10.5,4.5) (10,4) (6,2) (5.5,1.5)};

\draw  [dashed]  (10.5,2) --(5.5,2);
\node[ right] at (10.4,2) {$S_6$};

\draw [dashed] plot [smooth] coordinates {(12,4.5) (10,4) (0,2) (-1,2)};
\node[ right] at (12,4.5) {$S_3$};

\draw [dashed] plot [smooth] coordinates {(12,6.5) (10,6) (0,2) (-1,1.5)};
\node[right] at (12,6.5) {$S_1$};

\draw [dashed] plot [smooth] coordinates {(12,2.5) (10,2) (5,1) (0,2) (-1,2.5)};
\node[right] at (12,2.5) {$S_5$};

\node[right, below] at (11.2,6) {$(0:0:1)$};
\node[right, below] at (11.2,4) {$(i:0:1)$};
\node[right, below] at (11.4,2) {$(-i:0:1)$};
\node[right, below] at (11.2,0) {$(0:1:0)$};

\end{tikzpicture}
\label{fig:sections}
\end{figure}

\subsection{Invariant lattices}

Prior to computing the invariant lattice, we show a minimal set of generators for $\Pic(X_6)$ and an embedding of the lattice $T_{4,4,4}$ into $\Pic (X_6)$. This will be extremely useful in the computations of the invariant lattice for the two automorphisms $\sigma$ and $\sigma'$, that we give in Sections~\ref{ss:sigma} and \ref{ss:sigma'}. 

First, to find generators for $\Pic(X_6)$, we construct the incidence graph of the curves depicted in  Figure~\ref{fig:2nd_fibration} and look for an appropriate set of generators. This incidence graph is depicted in  Figure~\ref{fig:PicX_6Graph1}.  
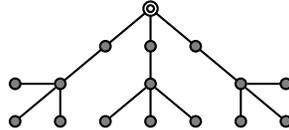
\begin{figure}[ht]
\caption{Configuration of curves from the elliptic fibration $\cE$. The vertex with double circle is the section.}
\centering
	
\begin{tikzpicture}[xscale=.6,yscale=.5, thick]
	
	\node[circle, draw, double, inner sep=0pt, minimum width=4pt, name=cw] at (0,0) {};

\begin{scope}[every node/.style={circle, draw, fill=black!50, inner sep=0pt, minimum width=4pt}]
	\draw (cw) -- (-1,-1) node{}
	-- (-2,-2) node[name=e1]{}
	-- (-2,-3) node{};
	\draw (-3,-2) node{} -- (e1) -- (-3,-3) node{};
	
	\draw (cw) -- (0,-1) node{}
	-- (0,-2) node[name=e2]{}
	-- (0,-3) node{};
	\draw (-1,-3) node{} -- (e2) -- (1,-3) node{};
	
	\draw (cw) -- (1,-1) node{}
	-- (2,-2) node[name=e3]{}
	-- (2,-3) node{};
	\draw (3,-2) node{} -- (e3) -- (3,-3) node{};

\end{scope}	
	
\end{tikzpicture}
\label{fig:PicX_6Graph1}
\end{figure}

As before, we set up the incidence matrix for this graph, and we can see that of these 16 curves, two are redundant. One can check that two of the vertices of valence 1 are redundant, as long as they are taken from different branches. For example, $\Pic (X_6)$ is generated by the curves pictured in Figure~\ref{fig:PicX_6Graph}.

\begin{figure}[ht]
	\caption{Curves from the elliptic fibration $\cE$. The filled vertices form a minimal set of generators for $\Pic (X_6)$.}
	\centering
	
\begin{tikzpicture}[xscale=.6,yscale=.5, thick]
	
	\node[circle, draw, inner sep=0pt, fill=black!50, minimum width=4pt, name=cw] at (0,0) {};

\begin{scope}[every node/.style={circle, draw, fill=black!50, inner sep=0pt, minimum width=4pt}]
	\draw (cw) -- (-1,-1) node{}
	-- (-2,-2) node[name=e1]{}
	-- (-2,-3) node{};
	\draw (e1) -- (-3,-3) node{};
	
	\draw (cw) -- (0,-1) node{}
	-- (0,-2) node[name=e2]{}
	-- (0,-3) node{};
	\draw (-1,-3) node{} -- (e2) -- (1,-3) node{};
	
	\draw (cw) -- (1,-1) node{}
	-- (2,-2) node[name=e3]{}
	-- (2,-3) node{};
	\draw (e3) -- (3,-3) node{};

\end{scope}

	\node[circle, draw, inner sep=0pt, minimum width=4pt, name=v3] at (3,-2) {};
	\node[circle, draw, inner sep=0pt, minimum width=4pt, name=v1] at (-3,-2) {};
	\draw[dotted] (v1) -- (e1);
	\draw[dotted] (v3) -- (e3);

	\end{tikzpicture}
	\label{fig:PicX_6Graph}
\end{figure}
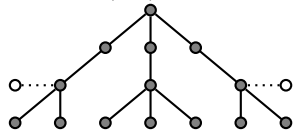

Furthermore, we can see the $T_{4,4,4}$ embedded primitively in that graph, for example, as in Figure~\ref{fig:T4embedding}. We see from this computation that given 7 rational curves forming the $IV^*$ and 3 more disjoint sections meeting the $IV^*$ in different components (in order to complete the three branches) one recognizes the $T_{4,4,4}$. Thus we see $T_{4,4,4}$ embedded primitively into $\Pic (X_6)$. We will use this primitive embedding for computing the invariant lattices of $\sigma$ and $\sigma'$.

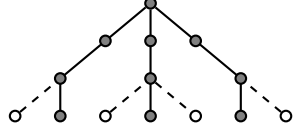
\begin{figure}[ht]
	\caption{$T_{4,4,4}$ embeds in $\Pic (X_6)$}
	\centering
	
	\begin{tikzpicture}[xscale=.6,yscale=.5, thick]
	
	\node[circle, draw, inner sep=0pt, fill=black!50, minimum width=4pt, name=cw] at (0,0) {};

	\begin{scope}[every node/.style={circle, draw, fill=black!50, inner sep=0pt, minimum width=4pt}]
	\draw (cw) -- (-1,-1) node{}
	-- (-2,-2) node[name=e1]{}
	-- (-2,-3) node{};
	
	\draw (cw) -- (0,-1) node{}
	-- (0,-2) node[name=e2]{}
	-- (0,-3) node{};

	\draw (cw) -- (1,-1) node{}
	-- (2,-2) node[name=e3]{}
	-- (2,-3) node{};

	\end{scope}

	\node[circle, draw, inner sep=0pt, minimum width=4pt, name=v3] at (3,-3) {};
	\node[circle, draw, inner sep=0pt, minimum width=4pt, name=v1] at (-3,-3) {};
	\node[circle, draw, inner sep=0pt, minimum width=4pt, name=y1] at (-1,-3) {};
	\node[circle, draw, inner sep=0pt, minimum width=4pt, name=y2] at (1,-3) {};
	\draw[dashed] (v1) -- (e1);
	\draw[dashed] (v3) -- (e3);
	\draw[dashed] (y1) -- (e2) -- (y2);

	\end{tikzpicture}
	\label{fig:T4embedding}
\end{figure}

\subsubsection{Invariant lattice for $\sigma_8$} 

We first give names to components of the $IV^*$ fiber in Figure~\ref{fig:sections}. 
The central component is called $\cE_0$ and it meets the components $\cE_1,\cE_2,\cE_3$. 
The external components are $\cE_4,\cE_5,\cE_6$ and $\cE_4$ meets $\cE_1$, $\cE_5$ meets $\cE_2$ and 
$\cE_6$ meets $\cE_3$. Let $\cE_6$ be the component meeting the zero section $S_0$ and $\cE_4$ (resp. $\cE_5$) be the component meeting the section $S_1$ (resp. $S_2$).

Notice that 
\[
\sigma^2=(\sigma')^2:(x,y,t)\mapsto(x,-y,\xi^5_8t),
\] 
where $\xi_8$ is a primitive 8th root of unity. This is non--symplectic of order 8; let us denote it by $\sigma_8$. In order to compute the invariant lattices $S(\sigma)$ and $S(\sigma')$, we will first compute the lattice $S(\sigma_8)$. 

We observe that $\sigma_8$ fixes the fibers over $t=0, t=\infty$; thus the curve $\cC$ is globally preserved as well as the fiber $IV^*$. The central component $\cE_0$ is fixed by $\sigma_8$ and the action of $\sigma_8$ on $\cC$ is an involution. 
Furthermore, $(\sigma_8)^4$ has order 2, it fixes $\cC$ (a curve of genus 1) and 4 rationals curves. Thus this is case 11 studied in \cite[Theorem 3.3]{order_eight}, and so we know that $\rk S(\sigma_8)=10$.

We now consider the lattice generated by $\{\cE_0, \cE_1,\cE_2,\cE_3,\cE_4,\cE_5,\cE_6, S_0, S_1,S_4\}$, which forms a $T_{4,4,4}$. Since $\sigma_8$ leaves all of the curves in our graph invariant and both have the same rank, $S(\sigma_8)$ is an overlattice.

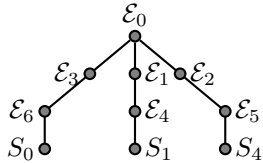
\begin{figure}[ht]
	\caption{Generators for invariant lattice $S(\sigma_8)$. This is the lattice $T_{4,4,4}$.}
	\centering
	
	\begin{tikzpicture}[xscale=.6,yscale=.5, thick]
	
	\node[circle, draw, inner sep=0pt, fill=black!50, minimum width=4pt, name=e0] at (0,0) {};
	
	\begin{scope}[every node/.style={circle, draw, fill=black!50, inner sep=0pt, minimum width=4pt}]
	\draw (e0) -- (-1,-1) node{}
	-- (-2,-2) node[name=e1]{}
	-- (-2,-3) node{};
	
	\draw (e0) -- (0,-1) node{}
	-- (0,-2) node[name=e2]{}
	-- (0,-3) node{};

	\draw (e0) -- (1,-1) node{}
	-- (2,-2) node[name=e3]{}
	-- (2,-3) node{};

	\end{scope}	
	
	\node [above] at (0,0){$\cE_0$};
		
		\node [left] at (-1,-1){$\cE_3$};
		\node [left] at (-2,-2){$\cE_6$};
		\node [left] at (-2,-3){$S_0$};

		\node [right] at (0,-1){$\cE_1$};
		\node [right] at (0,-2){$\cE_4$};
		\node [right] at (0,-3){$S_1$};
		
		\node [right] at (1,-1){$\cE_2$};
		\node [right] at (2,-2){$\cE_5$};
		\node [right] at (2,-3){$S_4$};
	
		\end{tikzpicture}

\end{figure}

Observe that we are taking as generators the components of the $IV^*$ fiber of the elliptic fibration $\mathcal{F}$ and three disjoint sections. 
Recall that $\Pic(X_6)$ is generated by the components of the $IV^*$ fiber, the zero section and 6 sections of infinite order. 
So we get clearly a primitive embedding of $T_{4,4,4}$ into $\Pic(X_6)$ and $S(\sigma_8)=T_{4,4,4}$. 

We will  use this invariant lattice to exhibit a primitive embedding for generators of the invariant lattices $S(\sigma)$. 
In the following computations, we will exhibit a lattice $L_\cB$ with an obvious primitive embedding 
\[
L_\cB \hookrightarrow S(\sigma_8).
\]
Since $S(\sigma_8)$ embeds primitively into $\Pic(X_6)$, we can conclude that $L_\cB$ does as well. As in Method IV from \cite{CP}, we will conclude that $L_\cB$ is the invariant lattice in question. 

Observe that in order to do the same for $\sigma'$, 
one needs to consider a different minimal set of generators for the invariant lattice $S(\sigma_8)$. The set $\{\cE_0, \cE_1,\cE_2,\cE_3,\cE_4,\cE_5,\cE_6, S_0, S_3,S_5\}$.
generates $S(\sigma_8)$ and serves the purpose.

\subsubsection{Invariant lattice for $\sigma$} \label{ss:sigma}

Let $L_\cB$ be the lattice generated by $\cB=\{\cE_0, \ldots, \cE_6, S_0, S_1,S_2\}$ the set of divisors fixed by $\sigma$. This is pictured in Figure~\ref{fig:X_6InvariantGen}. Observe that $S_1$ and $S_2$ meet in $(0:0:1)$ with multiplicity 4.

\begin{figure}[ht]
\caption{Generators for $L_\cB$}
\centering
\begin{tikzpicture}[xscale=.6,yscale=.5, thick]
	
	\node[circle, draw, fill=black!50, inner sep=0pt, minimum width=4pt, name=cw] at (0,0) {};
	
	\node[circle, draw, fill=black!50, inner sep=0pt, minimum width=4pt, name=cy] at (2,-3) {};	
	
\begin{scope}[every node/.style={circle, draw, fill=black!50, inner sep=0pt, minimum width=4pt}]
	\draw (cw) -- (-1,-1) node{}
	-- (-2,-2) node{}
	-- (-2,-3) node[name=e1]{};
	
	\draw (cw) -- (0,-1) node{}
	-- (0,-2) node{}
	-- (0,-3) node[name=e2]{}
	-- (cy);
	
	\draw (cw) -- (1,-1) node{}
	-- (2,-2) node[name=e3]{}
	-- (cy);
\end{scope}	
\node [above] at (0,0){$\cE_0$};
		
		\node [left] at (-1,-1){$\cE_3$};
		\node [left] at (-2,-2){$\cE_6$};
		\node [below] at (-2,-3){$S_0$};

		\node [right] at (0,-1){$\cE_1$};
		\node [right] at (0,-2){$\cE_4$};
		\node [below] at (0,-3){$S_1$};
		
		\node [right] at (1,-1){$\cE_2$};
		\node [right] at (2,-2){$\cE_5$};
		\node [below] at (2,-3){$S_2$};

	\node[below] at (1, -3){\small 4};
\end{tikzpicture}
\label{fig:X_6InvariantGen}
\end{figure}

Again, we look at the incidence matrix for this configuration, and we get the minimal set of generators of $L_\cB$ which is $\{\cE_0, \ldots, \cE_6, S_0, S_1\}$ as depicted in Figure~\ref{fig:X_6InvariantBasis}. Notice $S_2$ was redundant. 

\begin{figure}[ht]
\caption{Generators for the invariant lattice $S(\sigma)$}
	\centering
	\begin{tikzpicture}[xscale=.6,yscale=.5, thick]
	
	\node[circle, draw, fill=black!50, inner sep=0pt, minimum width=4pt, name=cw] at (0,0) {};
	\node[circle, draw, inner sep=0pt, minimum width=4pt, name=cy] at (2,-3) {};	

	\begin{scope}[every node/.style={circle, draw, fill=black!50, inner sep=0pt, minimum width=4pt}]
	\draw (cw) -- (-1,-1) node{}
	-- (-2,-2) node{}
	-- (-2,-3) node[name=e1]{};
	
	\draw (cw) -- (0,-1) node{}
	-- (0,-2) node{}
	-- (0,-3) node[name=e2]{};
		
	\draw (cw) -- (1,-1) node{}
	-- (2,-2) node[name=e3]{};
	
\end{scope}	
	
	\draw[dotted] (e2) -- (cy);
	\draw[dotted] (e3) -- (cy);
	
\end{tikzpicture}
\label{fig:X_6InvariantBasis}
\end{figure}
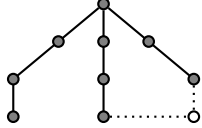  

One sees that it is a lattice of type $T_{3,4,4}$. The discriminant group of $T_{3,4,4}$ is $\mathbb Z/8\mathbb Z$ and the corresponding quadratic form is $\omega^{5}_{2,3}$. This clearly embeds primitively into $S(\sigma_8)$, and therefore the invariant lattice is $S(\sigma)\cong T_{3,4,4}$.

\subsubsection{Invariant lattice for $\sigma'$} \label{ss:sigma'}
The automorphism $\sigma'$ is 
\[\sigma'(x,y,t)=(-x,-iy,\zeta^5t),
\]
it has 6 fixed points and fixes no curves; it corresponds to the second line of Table~\ref{tab:thm3.1}. From the Table, we can conclude that the rank of the invariant lattice $S(\sigma')$ is 7. 

In order to compute the invariant lattice $S(\sigma')$, first observe that the automorphism is a reflection on the fiber of type $IV^*$ and interchanges the sections $S_3$ and $S_6$, and so we get three orbits of invariant curves.
We label the exceptional curves as before. Consider the set constructed from these orbits 
\[
\cB=\{\cE_0, \cE_3,\cE_6, \cE_1+\cE_2,\cE_4+\cE_5, S_0, S_3+S_6\},
\]
and let $L_\cB$ be the lattice generated by this set of divisors. The incidence graph for this set is depicted in Figure~\ref{fig:graphsigma2}. 

\begin{figure}[ht]
	\caption{Generators for $L_\cB$}
	\centering
\begin{tikzpicture}[xscale=.6,yscale=.5, thick]
	
	\node[circle, draw, fill=black!50, inner sep=0pt, minimum width=4pt, name=e0] at (0,0) {};

\begin{scope}[every node/.style={rectangle, draw, fill=black!50, inner sep=0pt, minimum width=4pt,minimum height=4pt}]
	
	\draw (e0) -- (0,-1) node{}
	-- (0,-2) node{}
	-- (0,-3) node[name=e1]{};
	
\end{scope}
	
\begin{scope}[every node/.style={circle, draw, fill=black!50, inner sep=0pt, minimum width=4pt}]
	
	\draw (e0) -- (1,-1) node{}
	-- (2,-2) node[name=e3]{}
	-- (2,-3) node{};	
\end{scope}	
	
	\node[left] at (0,-0.5) {2};
	\node[left] at (0,-1.5) {2};
	\node[left] at (0,-2.5) {2};
	
\end{tikzpicture}
	\label{fig:graphsigma2}
\end{figure}  

The lattice has rank 7 and so the generators in $\cB$ form a minimal set of generators of this lattice.
The invariant lattice $S(\sigma')$ is an overlattice, but we can clearly find a primitive embedding $L_\cB\hookrightarrow S(\sigma_8)$; therefore $L_\cB=S(\sigma')$. Using computations similar to those described previously, this lattice has rank 7 and discriminant quadratic form $v+\omega_{2,3}^{-1}$, and so this gives us $S(\sigma')=U\oplus D_4\oplus \langle -8\rangle$. 

\bibliographystyle{plain}
\bibliography{references}

\end{document}